\documentclass[11pt,reqno]{amsart}

\usepackage{amsmath,amssymb,amsfonts,amsthm,latexsym,enumerate,tikz,hyperref}

\newcommand{\Sym}{\mathrm{Sym}}
\newcommand{\Sy}{S}

\newcommand{\Aut}{\textup{Aut}}
\newcommand{\Bbox}{\,\Box\,}
\newcommand{\aut}{\mathrm{Aut}\,}

\newcommand{\id}{\mathrm{id}}

\newcommand{\bcp}{\mathcal{P}\,}
\newcommand{\TFA}{\mathrm{TFA}}

\begin{document}

\newtheorem{theorem}{Theorem}[section]
\newtheorem{lemma}[theorem]{Lemma}
\newtheorem{corollary}[theorem]{Corollary}
\newtheorem{proposition}[theorem]{Proposition}
\newtheorem{conjecture}[theorem]{Conjecture}
\theoremstyle{definition}
\newtheorem{definition}[theorem]{Definition}
\newtheorem{problem}[theorem]{Problem}
\newtheorem{claim}[theorem]{Claim}
\newtheorem{question}[theorem]{Question}
\newtheorem{remark}[theorem]{Remark}

\date{}

\title[Existence of unexpected automorphisms in direct product graphs]{The existence of unexpected automorphisms in direct product graphs}

\author{Xiaomeng Wang}
\address{School of Mathematics and Statistics, Lanzhou University, Lanzhou, Gansu 730000, P.\,R. China}
\email{wangxiaomeng@lzu.edu.cn}

\author[Qin]{Yan-Li Qin\,*}
\address{School of Statistics\\Capital University of Economics and Business\\Beijing, 100070\\P.R. China}
\email{ylqin@cueb.edu.cn}

\author{Binzhou Xia}
\address{School of Mathematics and Statistics, The University of Melbourne, Parkville, VIC 3010, Australia}
\email{binzhoux@unimelb.edu.au}

\thanks{* Corresponding author}

\maketitle

\begin{abstract}

A pair of graphs $(\Gamma,\Sigma)$ is called unstable if their direct product $\Gamma\times\Sigma$ admits automorphisms not from $\mathrm{Aut}(\Gamma)\times\mathrm{Aut}(\Sigma)$, and such automorphisms are said to be unexpected. The stability of a graph $\Gamma$ refers to that of $(\Gamma,K_2)$. While the stability of individual graphs has been relatively well studied, much less is known for graph pairs. In this paper, we propose a conjecture that provides the best possible reduction of the stability of a graph pair to the stability of a single graph. We prove one direction of this conjecture and establish partial results for the converse. This enables the determination of the stability of a broad class of graph pairs, with complete results when one factor is a cycle.


\end{abstract}

\section{Introduction}\label{sec1}

In this paper, all graphs are finite and undirected, without multiple edges but possibly with loops.
Abusing notation, we sometimes use $=$ instead of $\cong$ when there is no confusion, as we identify graphs or groups following the isomorphism.

Let $\Gamma$ and $\Sigma$ be graphs. The \emph{direct product} $\Gamma\times\Sigma$ is the graph with vertex set $V(\Gamma)\times V(\Sigma)$ in which $(u,i)$ is adjacent to $(v,j)$ if and only if $u$ is adjacent to $v$ in $\Gamma$ and $i$ is adjacent to $j$ in $\Sigma$. By definition, there is a natural embedding
\begin{equation}\label{eq:direct}
\Aut(\Gamma)\times\Aut(\Sigma)\leq\Aut(\Gamma\times\Sigma),
\end{equation}
where the product $\Aut(\Gamma)\times\Aut(\Sigma)$ acts componentwise. An automorphism $\tau\in\Aut(\Gamma\times\Sigma)$ is called \emph{unexpected} if $\tau\notin \Aut(\Gamma)\times\Aut(\Sigma)$. We say that the pair $(\Gamma,\Sigma)$ is \emph{unstable} if $\Gamma\times\Sigma$ admits an unexpected automorphism, and \emph{stable} otherwise. In other words, $(\Gamma,\Sigma)$ is stable if and only if \eqref{eq:direct} is an equality.

The concept of stability for graph pairs was introduced in~\cite{QXZZ2021} as a generalization of the earlier notion of stability for a single graph defined in~\cite{MSZ1989}.
Indeed, a graph $\Gamma$ is stable in the sense of~\cite{MSZ1989} if and only if the pair $(\Gamma,K_2)$ is stable.
The stability of graphs and graph pairs has since attracted considerable attention; see, for example,~\cite{FH2022,HMM2021,LMS2015,MSZ1989,MSZ1992,Morris2021,NS1996,QXZ2019,QXZ2021,Surowski2001,Surowski2003,Wilson2008} for graphs and~\cite{GLX2025,QXZ2024,QXZZ2021,WXZ2025} for graph pairs.

A graph is called \emph{twin-free} if no two vertices have the same neighborhood in the graph. We say that graphs $\Gamma$ and $\Sigma$ are \emph{coprime} (with respect to the direct product) if there exists no graph $\Delta$ of order greater than $1$ such that $\Gamma\cong\Gamma'\times\Delta$ and $\Sigma\cong\Sigma'\times\Delta$ for some graphs $\Gamma'$ and $\Sigma'$. If $\Gamma$ and $\Sigma$ are coprime connected non-bipartite twin-free graphs, then by a classic result~\cite[Theorem~8.18]{HIK2011}, the graph pair $(\Gamma,\Sigma)$ is stable. This together with some obvious necessary conditions given in~\cite{QXZZ2021} for a graph pair to be stable motivates the following definition of nontrivial graph pairs, so that we can focus on such graph pairs when considering their stability.

\begin{definition}\label{def:nontrivial}
A graph pair $(\Gamma,\Sigma)$ is \emph{nontrivial} if $\Gamma$ and $\Sigma$ are coprime connected twin-free graphs and exactly one of them is bipartite.
\end{definition}

A graph pair $(\Gamma,\Sigma)$ is said to be \emph{nontrivially unstable} if it is both nontrivial and unstable.
Explicitly determining which nontrivial pairs $(\Gamma,\Sigma)$ are stable is out of reach---even in the special case $\Sigma=K_2$~\cite{FH2022,HMM2021,QXZ2019,QXZ2021,Surowski2001,Wilson2008}.
Thus, the strongest results available are typically reductions of the stability problem for a nontrivial pair $(\Gamma,\Sigma)$, for instance with $\Sigma$ bipartite, to the stability of the single graph $\Gamma$.
As an example, the following question was raised in~\cite{QXZ2024}.

\begin{question}[{\cite[Question~19]{QXZ2024}}]\label{que}
Let $(\Gamma,\Sigma)$ be a nontrivial graph pair with $\Sigma$ bipartite. Under what condition does it hold that $(\Gamma, \Sigma)$ is stable if and only if $\Gamma$ is stable?
\end{question}

In this paper we formulate and investigate the following conjecture. If true, this conjecture represents the ultimate goal in the study of stability of graph pairs and shows that, surprisingly, the answer to Question~\ref{que} is ``no condition needed''.

\begin{conjecture}\label{conj}
Let $(\Gamma,\Sigma)$ be a nontrivial graph pair without loops, where  $\Sigma$ is bipartite. Then $(\Gamma,\Sigma)$ is stable if and only if $\Gamma$ is stable.
\end{conjecture}

To approach Conjecture~\ref{conj}, a common strategy is to analyze the structure of graph pairs via the Cartesian product and related notions.
The \emph{Cartesian product} $\Gamma\Bbox\Sigma$ of graphs $\Gamma$ and $\Sigma$ is the graph with vertex set $V(\Gamma)\times V(\Sigma)$ in which $(u,i)$ is adjacent to $(v,j)$ if and only if either $u=v$ and $i$ is adjacent to $j$ in $\Sigma$, or $i=j$ and $u$ is adjacent to $v$ in $\Gamma$.
We say that graphs $\Gamma$ and $\Sigma$ are \emph{Cartesian-coprime} if there exists no graph $\Delta$ of order greater than $1$ such that $\Gamma\cong\Gamma'\Bbox\Delta$ and $\Sigma\cong\Sigma'\Bbox\Delta$ for some graphs $\Gamma'$ and $\Sigma'$.
A graph with order greater than $1$ is called \emph{Cartesian-prime} if it cannot be represented as a Cartesian product of two graphs with smaller orders.
Furthermore, when studying the stability of a graph pair $(\Gamma,\Sigma)$, the Cartesian skeletons (see Definition~\ref{def:skeleton}) $S(\Gamma)$ and $S(\Sigma)$ have played an important role~\cite{GLX2025,HIK2011,QXZ2019}.
%In particular, if $\Sigma$ is connected and bipartite, then $S(\Sigma)$ has precisely two connected components, whose respective vertex sets are the two partite sets of $\Sigma$; see Lemma~\ref{le:lecsb}.

We now state the main result of this paper, which provides evidence in support of Conjecture~\ref{conj}.

\begin{theorem}\label{thm3}
Let $(\Gamma,\Sigma)$ be a nontrivial graph pair such that $\Sigma$ is bipartite with bipartition $\{U,W\}$. Then the following statements hold:
\begin{enumerate}[\rm(a)]
\item\label{thm3a} If $(\Gamma,\Sigma)$ is stable, then $\Gamma$ is stable.
\item\label{thm3b} Suppose that $S(\Gamma)$ is Cartesian-coprime to each connected component of $S(\Sigma)$, and there exists an automorphism of $\Sigma$ interchanging $U$ and $W$. If $\Gamma$ is stable, then $(\Gamma,\Sigma)$ is stable.
\item\label{thm3c} Suppose that both $U$ and $W$ have size at least three, no connected component of $B(\Sigma)$ is complete, each connected component of $S(\Sigma)$ is Cartesian-prime, and $S(\Sigma)$ admits an automorphism fixing some vertex in $U$ but not any in $W$. If $\Gamma$ is stable, then $(\Gamma,\Sigma)$ is stable.
\end{enumerate}
\end{theorem}

From Theorem~\ref{thm3}\,\eqref{thm3a} we already know that the ``only if'' part of Conjecture~\ref{conj} holds. Moreover, statements~\eqref{thm3b} and~\eqref{thm3c} show that, under certain conditions, the ``if'' part of Conjecture~\ref{conj} is also valid. Although these conditions may appear somewhat technical, they are sufficiently general to determine the stability of many graph pairs, particularly when one of the graphs is ``simple''. For example, in this paper we use Theorem~\ref{thm3} to determine the stability of graph pairs involving a cycle.

Before presenting our result, let us recall that the only previously published support for the ``if'' part of Conjecture~\ref{conj} is given in~\cite[Theorem~1.6]{GLX2025}, due to Gan, Liu and the third author. Their theorem is too technical to state here, but as a byproduct they proved that there is no nontrivially unstable graph pair involving a complete graph, thereby solving a problem from~\cite{DMS2019} and confirming a conjecture from~\cite{QXZ2024}. They also obtained the following result on the stability of graph pairs involving an odd cycle.

\begin{theorem}[{\cite[Theorem~1.8]{GLX2025}}]\label{thm1}
There exists no nontrivially unstable graph pair $(\Gamma,C_n)$ for odd $n$.
\end{theorem}

After the publication of~\cite{GLX2025}, we identified a gap in the proof of Theorem~\ref{thm1} (see Subsection~\ref{sec4} for details). In addition, it is natural to consider the stability of graph pairs involving even cycles. For instance, the following question has been raised in the literature.

\begin{question}[{\cite[Question~18]{QXZ2024}}]\label{que}
For a stable graph $\Gamma$ and an even integer $n\geq 6$, under what condition is $(\Gamma, C_n)$ nontrivially unstable?
\end{question}

As an application of Theorem~\ref{thm3}, we establish Theorem~\ref{le:lepco} below to answer Question~\ref{que}, and we strengthen Theorem~\ref{thm1} to the following Theorem~\ref{thm2} (which, in particular, furnishes a correct proof of Theorem~\ref{thm1}).
Note that a cycle $C_n$ is twin-free if and only if $n\neq4$.

\begin{theorem}\label{le:lepco}
Let $\Gamma$ be a graph, and let $n\geq6$ an even integer. Then $(\Gamma,C_n)$ is nontrivially unstable if and only if $\Gamma$ is nontrivially unstable and coprime to $C_n$.
\end{theorem}

\begin{theorem}\label{thm2}
Let $n\geq3$ be an integer such that $n\neq4$. Then there exists no nontrivially unstable graph pair $(\Gamma,C_n)$ if and only if $n$ is odd.
\end{theorem}

The remainder of the paper is organized as follows. After the preliminary section, we introduce and develop two key tools---two-fold automorphisms and partitions in Cartesian products of vertex sets---in Sections~\ref{sec5} and~\ref{sec6}, respectively. These tools not only play a central role in the proofs of the main results presented here, but also shed light on a general approach to Conjecture~\ref{conj}. Then in Section~\ref{sec2} we prove Theorem~\ref{thm3}, and in Section~\ref{sec3} we establish Theorems~\ref{le:lepco} and~\ref{thm2}.

\section{Preliminaries}

For a cycle $C_n$, we always label its vertices as $1,\ldots,n$ consecutively, counted modolo $n$.
Let $\Gamma$ be a graph. Denote the edge set of $\Gamma$ by $E(\Gamma)$ (note that $E(\Gamma)$ may contain singletons) and arc set of $\Gamma$ by $A(\Gamma)$ (an arc is just an ordered pair of adjacent vertices).
When we take arbitrary $\{u,v\}\in E(\Gamma)$, we do not assume $u\neq v$, as $\Gamma$ may have loops.
For a positive integer $n$, denote the Cartesian product of $n$ copies of $\Gamma$ by $\Gamma^{\Box,n}$, denote the cyclic group of order $n$ by $Z_n$, and if $n\neq6$ is even, then denote the dihedral group of order $n$ by $D_n$.

Recall the definition of Cartesian-prime graphs in the Introduction. For example, $C_n$ is Cartesian-prime for $n\neq4$.
An expression $\Gamma\cong\Gamma_1\Bbox\cdots\Bbox\Gamma_n$ with each $\Gamma_i$ Cartesian-prime is called a \emph{Cartesian-prime factorization} of the graph $\Gamma$.
%We say that graphs $\Gamma$ and $\Sigma$ are \emph{Cartesian-quasicoprime} if there exists no graph $\Delta$ of order greater than $1$ such that $\Gamma\cong\Gamma'\Bbox\Delta\Bbox\Delta$ and $\Sigma\cong\Sigma'\Bbox\Delta\Bbox\Delta$ for some graphs $\Gamma'$ and $\Sigma'$.
In the following lemma, statement~\eqref{thm:cartesianc} is easy to see, statement~\eqref{thm:cartesiand} can be found in~\cite[Theorem~6.21]{HIK2011}, statement~\eqref{thm:cartesiana} is given by Sabidussi and Vizing (see for instance~\cite[Theorem~6.6]{HIK2011}), and statement~\eqref{thm:cartesianb} can be derived from~\cite[Corollary~6.11]{HIK2011} and~\cite[Theorem~6.13]{HIK2011}.

\begin{lemma}\label{thm:cartesian}
The following statements hold:
\begin{enumerate}[\rm(a)]
\item\label{thm:cartesianc} If $\Gamma$ and $\Sigma$ are both connected graphs, then $\Gamma\Bbox\Sigma$ is connected.
\item\label{thm:cartesiand} If $\Gamma$, $\Sigma$ and $\Delta$ are graphs such that $\Gamma\Bbox\Delta\cong\Sigma\Bbox\Delta$ and $\Delta$ is nonempty, then $\Gamma\cong\Sigma$.
\item\label{thm:cartesiana} Every connected graph has a unique Cartesian-prime factorization up to isomorphism and the order of the factors.
\item\label{thm:cartesianb} The automorphism group of a Cartesian product of connected Cartesian-prime graphs is the direct product of the wreath products on the sets of pairwise isomorphic Cartesian-prime graphs.
\end{enumerate}
\end{lemma}

\begin{definition}\label{def:skeleton}
For a graph $\Gamma$, let $B(\Gamma)$ be the graph with vertex set $V(\Gamma)$ and edge set
\[
\{\{u,v\}:u, v\in V(\Gamma),\,u\neq v,\,\,N_{\Gamma}(u)\cap N_{\Gamma}(v)\neq\varnothing\}.
\]
An edge $\{u,v\}$ of $B(\Gamma)$ is said to be \emph{dispensable} with respect to $\Gamma$ if there exists some $w\in V(\Gamma)$ satisfying both of the conditions:
\begin{itemize}
\item $N_\Gamma(u)\cap N_\Gamma(v)\subsetneq N_\Gamma(u)\cap N_\Gamma(w)$ or $N_\Gamma(u)\subsetneq N_\Gamma(w)\subsetneq N_\Gamma (v)$, and
\item $N_\Gamma(v)\cap N_\Gamma(u)\subsetneq N_\Gamma(v)\cap N_\Gamma(w)$ or $N_\Gamma(v)\subsetneq N_\Gamma(w)\subsetneq N_\Gamma (u)$.
\end{itemize}
The {\em Cartesian skeleton} $S(\Gamma)$ of $\Gamma$ is the graph obtained from $B(\Gamma)$ by removing all dispensable edges with respect to $\Gamma$.
\end{definition}

For example, $B(C_n)$ is the graph with vertex set $\{1, \ldots, n\}$ such that each $i$ is adjacent to $i+2$ and $i-2$, whence $B(C_n)$ is a cycle of length $n$ when $n$ is odd, and is the union of two cycles of length $n/2$ when $n$ is even.
Moreover, there is no edge in $B(C_n)$ that is dispensable with respect to $C_n$. Therefore,
\begin{equation}\label{eq:S(Cn)}
S(C_n)=B(C_n)\cong
\begin{cases}
C_n&\text{if $n$ is odd}\\
2C_{n/2}&\text{if $n$ is even}.
\end{cases}
\end{equation}
The next lemma collects some basic facts about Cartesian skeletons for our needs.

\begin{lemma}[{\cite[Propositions~8.10~and~8.13]{HIK2011}}]\label{le:lecsb}
Let $\Gamma$ and $\Sigma$ be graphs. Then the following statements hold:
\begin{enumerate}[\rm(a)]
\item\label{le:leas} Each automorphism $\Gamma$ is also an automorphisms of $S(\Gamma)$.
\item\label{it:itcsba} If $\Gamma$ is connected and non-bipartite, then $S(\Gamma)$ is connected.
\item\label{it:itcsbb} If $\Gamma$ is connected and bipartite, then $S(\Gamma)$ has precisely two connected components, whose respective vertex sets are the two partite sets of $\Gamma$.
\item\label{le:lecsd} If $\Gamma$ and $\Sigma$ are both twin-free without isolated vertices, then $S(\Gamma\times \Sigma)=S(\Gamma)\Bbox S(\Sigma)$.
\end{enumerate}
\end{lemma}

The last result of this section, although slightly technical, are applied frequently in the paper.

\begin{lemma}\label{claim1}
Let $\Gamma$ and $\Sigma$ be connected twin-free graphs such that $\Gamma$ is non-bipartite and $\Sigma$ is bipartite with bipartition $\{U,W\}$, and let $\tau$ be an automorphism of $S(\Gamma)\Bbox S(\Sigma)$. Then there exists $\sigma\in\Aut(S(\Sigma))$ with $|\sigma|\leq2$ such that $(\mathrm{id},\sigma)\tau$ and $\tau(\mathrm{id},\sigma)$ stabilize $V(\Gamma)\times U$, where $\id$ is the identity permutation on $V(\Gamma)$, and that if $|\sigma|=2$ then $\sigma$ interchanges $U$ and $W$.
\end{lemma}

\begin{proof}
By Lemma~\ref{le:lecsb}, there are precisely two connected components of $S(\Sigma)$, say $S_+(\Sigma)$ and $S_-(\Sigma)$, with vertex sets $U$ and $W$ respectively. Let $X=V(\Gamma)\times U$ and $Y=V(\Gamma)\times W$.
Since $\tau$ is an automorphism of $S(\Gamma\times\Sigma)=S(\Gamma)\Bbox S(\Sigma)$, it preserves the bipartition $\{X,Y\}$.
If $\tau$ stabilizes $X$, then the identity permutation $\sigma$ on $V(\Sigma)$ is as required.
Now assume that $\tau$ interchanges $X$ and $Y$. Then $\tau$ induces a graph isomorphism between the two connected components
\[
S(\Gamma)\Bbox S_+(\Sigma)\ \text{ and }\ S(\Gamma)\Bbox S_-(\Sigma)
\]
of $S(\Gamma\times\Sigma)$ with vertex sets $X$ and $Y$ respectively. By Lemma~\ref{thm:cartesian}, $S_+(\Sigma)\cong S_-(\Sigma)$. Hence there exists an involution $\sigma\in\Aut(S(\Sigma))$ swapping $V(S_+(\Sigma))=U$ and $V(S_-(\Sigma))=W$, and it follows that
\[
\aut(S(\Gamma\times\Sigma))=(\aut(S(\Gamma)\Bbox S_+(\Sigma))\times\aut(S(\Gamma)\Bbox S_-(\Sigma)))\rtimes\langle(\mathrm{id},\sigma)\rangle.
\]
Since both $\tau$ and $(\mathrm{id},\sigma)$ interchanges $X$ and $Y$, the product $\tau(\mathrm{id},\sigma)$ stabilizes $X$. Hence $\sigma$ is as required.
\end{proof}

\section{Two-fold-automorphisms}\label{sec5}

\begin{definition}
For a graph $\Gamma$, a pair $(\alpha,\beta)$ of permutations on $V(\Gamma)$ is a \emph{two-fold-automorphism} (\emph{TF-automorphism} for short) if for $s,t\in V(\Gamma)$,
\begin{equation}\label{eq:TF}
(s,t)\in A(\Gamma)\,\Leftrightarrow\,(s^\alpha,t^\beta)\in A(\Gamma).
\end{equation}
Since $(\alpha,\beta)$ induces a permutation on $V(\Gamma)\times V(\Gamma)$ and $V(\Gamma)$ is finite,~\eqref{eq:TF} is equivalent to
\[
(s,t)\in A(\Gamma)\,\Rightarrow\,(s^\alpha,t^\beta)\in A(\Gamma).
\]
The set of TF-automorphisms of $\Gamma$ forms a group under the product
\[
(\alpha_1,\beta_1)(\alpha_2,\beta_2):=(\alpha_1\alpha_2,\beta_1\beta_2),
\]
and we denote this group by $\TFA(\Gamma)$. A TF-automorphism $(\alpha,\beta)$ is \emph{diagonal} if $\alpha=\beta$.
\end{definition}

A diagonal TF-automorphism $(\alpha,\alpha)$ corresponds uniquely to an automorphism $\alpha$, which gives $\Aut(\Gamma)\leq\TFA(\Gamma)$ for each graph $\Gamma$. Moreover, the equality holds if and only if every TF-automorphism of $\Gamma$ is diagonal. Writing $V(K_2)=\{0,1\}$, the graph $\Gamma\times K_2$ has a bipartition $\{V(\Gamma)\times\{0\},V(\Gamma)\times\{1\}\}$. The following lemma is from~\cite[Lemma~3.1~and~Theorem~3.2]{LMS2015}.

\begin{lemma}\label{lem:TF}
For a graph $\Gamma$, the following statements hold:
\begin{enumerate}[\rm(a)]
\item\label{lem:TF-c} The subgroup of $\Aut(\Gamma\times K_2)$ stabilizing $V(\Gamma)\times\{0\}$ is isomorphic to $\TFA(\Gamma)$.
\item\label{lem:TF-a} If $\Gamma$ has a non-diagonal TF-automorphism, then $\Gamma$ is unstable.
\item\label{lem:TF-b} Suppose that $\Gamma$ is connected and non-bipartite. Then $\Aut(\Gamma\times K_2)=\TFA(\Gamma)\rtimes Z_2$. In particular, $\Gamma$ is unstable if and only if it admits a non-diagonal TF-automorphism.
\end{enumerate}
\end{lemma}

The equality $|\Aut(\Gamma\times K_2)|=2|\TFA(\Gamma)|$ from Lemma~\ref{lem:TF}\,\eqref{lem:TF-b} might fail for bipartite graphs, as detailed in the next lemma.

\begin{lemma}\label{lem:TF-bipartite}
Let $\Gamma$ be a connected bipartite graph with bipartition $\{U,W\}$.
\begin{enumerate}[\rm(a)]
\item\label{lem:TF-bipartite-a} If $\Gamma$ admits an automorphism swapping $U$ and $W$, then $|\Aut(\Gamma\times K_2)|=4|\TFA(\Gamma)|$.
\item\label{lem:TF-bipartite-b} If $\Gamma$ admits no automorphism swapping $U$ and $W$, then $\Aut(\Gamma\times K_2)=\TFA(\Gamma)\rtimes Z_2$. In particular, in this case, $\Gamma$ is unstable if and only if it admits a non-diagonal TF-automorphism.
\end{enumerate}
\end{lemma}

\begin{proof}
Denote $\Pi=\Gamma\times K_2$ and $X=\Aut(\Pi)$, and for $i\in\{0,1\}=V(K_2)$, let $U_i=U\times\{i\}$ and $W_i=W\times\{i\}$. By Lemma~\ref{lem:TF}\,\eqref{lem:TF-c}, the stabilizer $Y:=X_{U_0\cup W_0}$ is isomorphic to $\TFA(\Gamma)$. Observe that, since $\Gamma$ is connected and bipartite, $\Pi$ has precisely two connected components, and they are the induced subgraphs of $\Pi$ on $U_0\cup W_1$ and $U_1\cup W_0$ respectively. Denote these two subgroups by $\Pi_+$ and $\Pi_-$  respectively. Then $\Pi_+\cong\Pi_-\cong\Gamma$, and
\begin{equation}\label{eq:Aut-bipartite-K2}
X=(\Aut(\Pi_+)\times\Aut(\Pi_-))\rtimes\langle\tau\rangle=\Aut(\Pi)\wr Z_2,
\end{equation}
where $\tau$ is the involution swapping $(v,0)$ and $(v,1)$ for each $v\in V(\Gamma)$. It follows that $\overline{V}:=\{U_0,W_0,U_1,W_1\}$ is an $X$-invariant partition of $V(\Pi)$. Let $\overline{\Pi}$ be the quotient graph of $\Pi$ induced by $\overline{V}$, as in Figure~\ref{fig1}, and let $\overline{X}$ and $\overline{Y}$ be the induced permutation group on $\overline{V}$ by $X$ and $Y$ respectively. Then $\overline{X}\leq\Aut(\overline\Pi)$, and $\overline{Y}=\overline{X}_{\{U_0,W_0\}}$ such that $|X:Y|=|\overline{X}:\overline{Y}|$.

\begin{figure}[h]
\begin{center}
\caption{The graph $\overline{\Pi}$}\label{fig1}
\medskip
\begin{tikzpicture}[every node/.style={font=\small}, line width=0.8pt, scale=0.7]
% vertices with coordinates + labels
\fill (-1.2,1.2) circle (0.08) node[left]  {$U_0$};
\fill ( 1.2,1.2) circle (0.08) node[right] {$U_1$};
\fill (-1.2,-1.2) circle (0.08) node[left]  {$W_0$};
\fill ( 1.2,-1.2) circle (0.08) node[right] {$W_1$};
% crossed edges
\draw (-1.2,1.2) -- ( 1.2,-1.2);
\draw ( 1.2,1.2) -- (-1.2,-1.2);
\end{tikzpicture}
\end{center}
\end{figure}

If $\Gamma$ admits an automorphism swapping $U$ and $W$, then~\eqref{eq:Aut-bipartite-K2} implies that $\overline{X}=Z_2\wr Z_2$, and $\overline{Y}=\overline{X}_{\{U_0,W_0\}}$ is generated by the involution swapping $U_0$ and $W_0$ as well as swapping $U_1$ and $W_1$.
In this case, $|X:Y|=|\overline{X}:\overline{Y}|=4$, and so $|X|=4|Y|=4|\TFA(\Gamma)|$, as part~\eqref{lem:TF-bipartite-a} states.
Now assume that $\Gamma$ admits no automorphism swapping $U$ and $W$. Then $\Aut(\Pi_+)$ stabilizes both $U_0$ and $W_1$, while $\Aut(\Pi_-)$ stabilizes both $U_1$ and $W_0$. Thus, by~\eqref{eq:Aut-bipartite-K2}, $\overline{X}=\langle\overline\tau\rangle$, where $\overline\tau$ is the permutation on $\overline{V}$ induced by $\tau$ (that is, $\overline\tau$ swaps $U_0$ and $U_1$ as well as swaps $W_0$ and $W_1$), and $\overline{Y}=1$. It follows that $|X:Y|=|\overline{X}:\overline{Y}|=2$, and
\[
X=Y\rtimes\langle\tau\rangle=\TFA(\Gamma)\rtimes Z_2,
\]
which proves part~\eqref{lem:TF-bipartite-b}.
\end{proof}

\begin{remark}
Only under the condition $|\Aut(\Gamma\times K_2)|=2|\TFA(\Gamma)|$ can we conclude that $\Gamma$ is unstable if and only if it admits a non-diagonal TF-automorphism. Note that this condition is omitted in the statement of~\cite[Theorem~3.2]{LMS2015}, and that the condition may fail---for example, when $\Gamma$ is a connected bipartite graph satisfying the assumption of Lemma~\ref{lem:TF-bipartite}\,\eqref{lem:TF-bipartite-a}.
\end{remark}

TF-automorphisms are naturally related to Cartesian skeletons, as stated in the next lemma.

\begin{lemma}\label{TF-Aut}
For a TF-automorphism $(\alpha,\beta)$ of a graph $\Gamma$, the following statements hold:
\begin{enumerate}[{\rm (a)}]
\item\label{TF-Aut-a} Both $\alpha$ and $\beta$ are automorphisms of $\Sy(\Gamma)$.
\item\label{TF-Aut-b} If $\Gamma$ is connected and bipartite, then $\alpha$ and $\beta$ either both stabilize or both interchange the two bipartite sets.
\end{enumerate}
\end{lemma}

\begin{proof}
Statement~\eqref{TF-Aut-a} follows from~\cite[Lemma~4.2(c)]{QXZ2019}. Now let $\Gamma$ be connected and bipartite, with bipartition $\{U,W\}$, say. We deduce from~\eqref{TF-Aut-a} that $\alpha$ and $\beta$ both preserve the partition $\{U,W\}$. To prove statement~\eqref{TF-Aut-b}, suppose for a contradiction that $\alpha$ stabilizes both $U$ and $W$ whereas $\beta$ swaps $U$ and $W$. For any $\{u,w\}\in E(\Gamma)$, where $u\in U$ and $w\in W$, it follows that $u^\alpha$ and $v^\beta$ are either both in $U$ or both in $W$. However, this implies that $u^\alpha$ and $v^\beta$ are not adjacent in $\Gamma$, contradicting the assumption that $(\alpha,\beta)$ is a TF-automorphism of $\Gamma$.
\end{proof}

We conclude this section with the following result, which is the key to the proof of Theorem~\ref{thm3}\,\eqref{thm3b}.

\begin{lemma}\label{lem:TFA}
Let $\Gamma$ be a connected twin-free non-bipartite graph, and let $\Sigma$ be a connected twin-free bipartite graph with parts $U$ and $W$, such that $S(\Gamma)$ is Cartesian-coprime to each connected component of $S(\Sigma)$, and there exists an automorphism of $\Sigma$ interchanging $U$ and $W$. Then for each TF-automorphism $(\alpha,\beta)$ of $\Gamma\times\Sigma$, there exist TF-automorphisms $(\alpha_+,\beta_-)$ and $(\alpha_-,\beta_+)$ of $\Gamma$ and TF-automorphism $(\alpha_0,\beta_0)$ of $\Sigma$ such that
\begin{align*}
\alpha|_{V(\Gamma)\times U}=(\alpha_+,\alpha_0)|_{V(\Gamma)\times U},&\ \
\alpha|_{V(\Gamma)\times W}=(\alpha_-,\alpha_0)|_{V(\Gamma)\times W},\\
\beta|_{V(\Gamma)\times U}=(\beta_+,\beta_0)|_{V(\Gamma)\times U},&\ \
\beta|_{V(\Gamma)\times W}=(\beta_-,\beta_0)|_{V(\Gamma)\times W}.
\end{align*}
\end{lemma}

\begin{proof}
Let $\Pi=\Gamma\times\Sigma$. By Lemma~\ref{le:lecsb}, $S(\Sigma)$ has precisely two connected components, say $S_+(\Sigma)$ and $S_-(\Sigma)$, with vertex sets $U$ and $W$ respectively, and $S(\Pi)=S(\Gamma)\Bbox S(\Sigma)$ has precisely two connected components
\begin{equation}\label{decomposition}
S_+(\Pi)=S(\Gamma)\Bbox S_+(\Sigma)\ \text{ and }\ S_-(\Pi)=S(\Gamma)\Bbox S_-(\Sigma).
\end{equation}
Since $S(\Gamma)$ is Cartesian-coprime to $S_+(\Sigma)$ and $S_-(\Sigma)$, we conclude by Lemma~\ref{thm:cartesian} that
\begin{equation}\label{eqS0}
\Aut( S_+(\Pi))=\Aut(S(\Gamma))\times\Aut( S_+(\Sigma)),
\end{equation}
\begin{equation}\label{eqS1}
\Aut( S_-(\Pi))=\Aut(S(\Gamma))\times\Aut( S_-(\Sigma)).
\end{equation}
Let $V=V(\Gamma)$, $X=V\times U$ and $Y=V\times W$. By Lemma~\ref{TF-Aut}, $\alpha$ and $\beta$ are automorphisms of $S(\Pi)$ such that either they both stabilize $X$ or they both interchange $X$ and $Y$. According to these two distinct cases, we construct automorphisms of $S(\Pi)$ stabilizing $X$ in the following two paragraphs.

First assume that $\alpha$ and $\beta$ stabilize $X$. Take $\sigma$ to be the identity permutation on $V(\Sigma)$,
\[
\alpha'=\alpha(\id,\sigma)=\alpha\ \text{ and }\ \beta'=\beta(\id,\sigma)=\beta,
\]
where $\id$ is the identity permutation on $V$. Then both $\alpha'$ and $\beta'$ are automorphisms of $S(\Pi)$ that stabilize $X$.

Now assume that $\alpha$ and $\beta$ swap $X$ and $Y$. Since $\alpha\in\Aut(S(\Pi))$, we obtain $S_+(\Pi)\cong S_-(\Pi)$. This together with~\eqref{decomposition} and Lemma~\ref{thm:cartesian}\,\eqref{thm:cartesiand} implies that $ S_+(\Sigma)\cong S_-(\Sigma)$. Take $\sigma$ to be an automorphism  of $\Aut(\Sigma)$ swapping $S_+(\Sigma)$ and $S_-(\Sigma)$,
\[
\alpha'=\alpha(\id,\sigma)\ \text{ and }\ \beta'=\beta(\id,\sigma),
\]
where $\id$ is the identity permutation on $V$. Then $\alpha'$ and $\beta'$ are automorphisms of $S(\Pi)$ stabilizing $X$.

To sum up the above two paragraphs, we have constructed automorphisms $\alpha'=\alpha(\id,\sigma)$ and $\beta'=\beta(\id,\sigma)$ of $S(\Pi)$ that stabilize $X$. It follows that $\alpha'$ and $\beta'$ induce automorphisms of $S_+(\Pi)$ and $S_-(\Pi)$, and so by~\eqref{eqS0} and~\eqref{eqS1}, there exist $\alpha_+,\alpha_-,\beta_+,\beta_-\in\Sym(V)$ and $\overline{\alpha},\overline{\beta}\in\Sym(V(\Sigma))$ such that
\begin{equation}\label{aut012}
\alpha'|_X=(\alpha_+, \overline{\alpha})|_X,\ \ \alpha'|_Y=(\alpha_-, \overline{\alpha})|_Y,\ \
\beta'|_X=(\beta_+, \overline{\beta})|_X,\ \ \beta'|_Y=(\beta_-, \overline{\beta})|_Y.
\end{equation}
Since $(\alpha,\beta)\in\TFA(\Pi)$ and $(\id,\sigma)\in\Aut(\Pi)$, we have $(\alpha',\beta')\in\TFA(\Pi)$.
For any fixed $\{u,w\}\in E(\Sigma)$ with $u\in U$ and $w\in W$,
\begin{align}
((s,u),(t,w))\in A(\Pi)&\,\Leftrightarrow\,((s,u)^{\alpha'},(t,w)^{\beta'})\in A(\Pi)
\,\Leftrightarrow\,((s^{\alpha_+},u^{\overline{\alpha}}),(t^{\beta_-},w^{\overline{\beta}}))\in A(\Pi),\label{eq:Qin1}\\
((s,w),(t,u))\in A(\Pi)&\,\Leftrightarrow\,((s,w)^{\alpha'},(t,u)^{\beta'})\in A(\Pi)
\,\Leftrightarrow\,((s^{\alpha_-},w^{\overline{\alpha}}),(t^{\beta_+},u^{\overline{\beta}}))\in A(\Pi)\label{eq:Qin2}.
\end{align}
Thus, for each $(s,t)\in A(\Gamma)$, both $(s^{\alpha_+},t^{\beta_-})$ and $(s^{\alpha_-},t^{\beta_+})$ are arcs of $\Gamma$. This shows that $(\alpha_+,\beta_-)$ and $(\alpha_-,\beta_+)$ are TF-automorphisms of $\Gamma$. Fixing any $\{s,t\}\in E(\Gamma)$, the relations~\eqref{eq:Qin1} and~\eqref{eq:Qin2} also indicate that $(i,j)\in A(\Sigma)\Rightarrow(i^{\overline{\alpha}},j^{\overline{\beta}})\in A(\Sigma)$, and so $(\overline{\alpha},\overline{\beta})$ is a TF-automorphism of $\Sigma$. Let $\alpha_0=\overline{\alpha}\sigma^{-1}$ and $\beta_0=\overline{\beta}\sigma^{-1}$.
Then $(\alpha_0,\beta_0)$ is a TF-automorphism of $\Sigma$.
Moreover, since $\alpha=\alpha'(\id,\sigma^{-1})$ and $\beta=\beta'(\id,\sigma^{-1})$, we derive from~\eqref{aut012} that
\[
\alpha|_X=(\alpha_+,\alpha_0)|_X,\ \ \alpha|_Y=(\alpha_-,\alpha_0)|_Y,\ \ \beta|_X=(\beta_+,\beta_0)|_X,\ \ \beta|_Y=(\beta_-,\beta_0)|_Y.
\]
This completes the proof.
\end{proof}

\section{Partitions in Cartesian product of vertex sets}\label{sec6}

For sets $U$ and $W$, the \emph{$U$-partition} and \emph{$W$-partition} of $U\times W$ are defined as the partitions
\[
\{\{u\}\times W:u\in U\}\ \text{ and }\ \{U\times \{w\}:w\in W\},
\]
respectively.

\begin{lemma}[{\cite[Lemma~3.1]{GLX2025}}]\label{le:leggd}
A nonempty graph $\Gamma$ is isomorphic to $\Delta\times\Sigma$ for some graphs $\Delta$ and $\Sigma$ if and only if $V(\Gamma)$ has partitions $\mathcal{P}$ and $\mathcal{L}$ such that the following conditions hold:
\begin{enumerate}[\rm(a)]
\item\label{it:itggda} $|P\cap L|=1$ for all $P\in\mathcal{P}$ and $L\in\mathcal{L}$;
\item\label{it:itggdb} for $P_u, P_v\in \mathcal{P}$ and $L_i,L_j\in\mathcal{L}$, if there exist edges of $\Gamma$ from $P_u$ to $P_v$ and $L_i$ to $L_j$, then there exists an edge of $\Gamma$ from $P_u\cap L_i$ to $P_v\cap L_j$.
\end{enumerate}
Moreover, if conditions~\eqref{it:itggda} and~\eqref{it:itggdb} hold, then $\Delta$ can be taken as the graph with vertex set $\mathcal{P}$ such that $\{P_u,P_v\}\in E(\Delta)$ whenever there exists an edge of $\Gamma$ between $P_u$ and $P_v$, and $\Sigma$ can be taken as the graph with vertex set $\mathcal{L}$ such that $\{L_i,L_j\}\in E(\Sigma)$ whenever there exists an edge of $\Gamma$ between $L_i$ and $L_j$.
\end{lemma}

The following lemma investigates the decomposition of a graph into direct product factors, one of which is a cycle.

\begin{lemma}\label{le:ledc}
Let $\Gamma$ be a connected graph, and let $n\geq 3$ be an integer. Then $\Gamma\cong\Delta\times C_n$ for some graph $\Delta$ if and only if $V(\Gamma)$ has a partition $\{L_1,\ldots,L_n\}$ into independent sets and a partition $\mathcal{P}$ such that the following conditions hold:
\begin{enumerate}[\rm(a)]
\item\label{it:itdca} $|P\cap L_i|=1$ for all $P\in\mathcal{P}$ and $i\in\{1,\ldots,n\}$;
\item\label{it:itdcb} for $P_u, P_v\in\bcp$ and $i,j\in \{1,\ldots,n\}$, if there exist edges of $\Gamma$ from $P_u$ to $P_v$ and $L_i$ to $L_j$, then there exists an edge of $\Gamma$ from $P_u\cap L_i$ to $P_v\cap L_j$;
\item\label{it:itdcc} for each given $P_u$ and $P_v$ in $\mathcal{P}$, the number $|N_\Gamma(x)\cap P_v|$ is a constant $d_{u,v}$ for all $x\in P_u$ such that $d_{u,v}\in\{0,2\}$.
\end{enumerate}
\end{lemma}

\begin{proof}
Suppose that $\Gamma\cong\Delta\times C_n$ for some graph $\Delta$. Without loss of generality, identify $\Gamma$ with $\Delta\times C_n$. For $u\in V(\Delta)$ and $i\in V(C_n)$, let
\[
P_u=\{(u,j):j\in V(C_n)\}\ \text{ and }\ L_i=\{(v,i): v\in V(\Delta)\}.
\]
Clearly, each $L_i$ is an independent set. Let $\bcp=\{P_u:u\in V(\Delta)\}$. It is straightforward to verify conditions~\eqref{it:itdca}--\eqref{it:itdcc}.

Conversely, suppose that $V(\Gamma)$ has a partition $\{L_1,\ldots,L_n\}$ into independent sets and a partition  $\mathcal{P}$ satisfying~\eqref{it:itdca}--\eqref{it:itdcc}. Let $\Delta$ be the digraph with vertex set $\mathcal{P}$ such that $\{P_u,P_v\}\in E(\Delta)$ if and only if there exists an edge of $\Gamma$ from $P_u$ to $P_v$. Similarly, Let $\Sigma$ be the digraph with vertex set $\{L_1,\ldots,L_n\}$ such that $\{L_i,L_j\}\in E(\Sigma)$ if and only if there exists an edge of $\Gamma$ from $L_i$ to $L_j$.
Then both $\Delta$ and $\Sigma$ are undirected, and $\Sigma$ has no loops as $L_1,\ldots,L_n$ are independent sets in $\Gamma$.
By Lemma~\ref{le:leggd}, we derive from conditions~\eqref{it:itdca} and~\eqref{it:itdcb} that $\Gamma\cong\Delta\times\Sigma$.
Without loss of generality, identify $\Gamma$ with $\Delta\times\Sigma$. Since $\Gamma$ is connected, both $\Delta$ and $\Sigma$ are connected. Take an arbitrary $i\in\{1,\ldots,n\}$ and take any edge $\{P_u,P_v\}$ of $\Delta$. Then the constant $d_{u,v}$ in condition~\eqref{it:itdcc} is nonzero and hence equals $2$. Let $x$ be the unique element in $P_u\cap L_i$. Then under the identification of $\Gamma$ with $\Delta\times\Sigma$, the vertex $x$ of $\Gamma$ is viewed as $(P_u,L_i)\in V(\Delta)\times V(\Sigma)$, and the condition $|N_\Gamma(x)\cap P_v|=d_{u,v}=2$ implies that the degree of $L_i$ in $\Sigma$ is $2$. Since $i$ is arbitrary and $\Sigma$ is connected without loops, it follows that $\Sigma\cong C_n$. Thus, $\Gamma\cong\Delta\times\Sigma\cong\Delta\times C_n$.
\end{proof}

For sets $U$ and $W$, a permutation of $U\times W$ is called a \emph{$U$-mixer} if it does not preserve the $U$-partition, and is called a \emph{$W$-mixer} if it does not preserve the $W$-partition. For a graph $\Gamma$, we simply call $V(\Gamma)$-partitions \emph{$\Gamma$-partitions} and $V(\Gamma)$-mixers \emph{$\Gamma$-mixers}.
In the following lemma, the first statement is obvious, while the second statement, which is exactly~\cite[Lemma~2.3]{GLX2025}, is not difficult to see from the first.

\begin{lemma}\label{le:legsm}
The following statements hold:
\begin{enumerate}[\rm(a)]
\item\label{le:legsm-a} For sets $U$ and $W$, a permutation of $U\times W$ is not in $\Sym(U)\times\Sym(W)$ if and only if it is either a $U$-mixer or a $W$-mixer.
\item\label{le:legsm-b} For nonempty graphs $\Gamma$ and $\Sigma$, an automorphism of $\Gamma\times\Sigma$ is unexpected if and only if it is either a $\Gamma$-mixer or a $\Sigma$-mixer.
\end{enumerate}
\end{lemma}

The rest of this section is devoted to Proposition~\ref{lem:mixer}. Recall from Lemma~\ref{le:lecsb} that $S(\Gamma)$ is connected for each connected bipartite graph $\Gamma$, whereas if $\Sigma$ is a connected bipartite graph with bipartition $\{U,W\}$, then $S(\Sigma)$ has precisely two connected components with vertex sets $U$ and $W$ respectively. Then by Lemma~\ref{thm:cartesian}, $S(\Gamma)\Bbox S(\Sigma)$ has precisely two connected components, whose vertex sets are $V(\Gamma)\times U$ and $V(\Gamma)\times W$ respectively.

\begin{lemma}\label{lem:Gamma-mixer}
Let $\Gamma$ and $\Sigma$ be connected graphs such that $\Gamma$ is non-bipartite and $\Sigma$ is bipartite with bipartition $\{U,W\}$, and let $\rho$ be an automorphism of $S(\Gamma)\Bbox S(\Sigma)$ that stabilizes $V(\Gamma)\times U$ and is a $\Sigma$-mixer of $V(\Gamma)\times V(\Sigma)$. Then $\rho$ is a $\Gamma$-mixer of $V(\Gamma)\times V(\Sigma)$.
\end{lemma}

\begin{proof}
Let $S_+(\Sigma)$ and $S_-(\Sigma)$ be the induced subgraphs of $S(\Sigma)$ by $U$ and $W$, respectively, let $V=V(\Gamma)$, and let $\pi$ be the projection of $V\times V(\Sigma)$ to $V$. By Lemmas~\ref{thm:cartesian} and~\ref{le:lecsb}, $S(\Gamma)\Bbox S_+(\Sigma)$ and $S(\Gamma)\Bbox S_-(\Sigma)$ are the connected components of $S(\Gamma)\Bbox S(\Sigma)$.
Hence $\rho$ induces automorphisms on $S(\Gamma)\Bbox S_+(\Sigma)$ and $S(\Gamma)\Bbox S_-(\Sigma)$, say $\phi$ and $\psi$, respectively.

Since $\rho$ is a $\Sigma$-mixer of $V\times(U\cup W)$, either $\phi$ is a $U$-mixer of $V\times U$, or $\psi$ is a $W$-mixer of $V\times W$. Without loss of generality, assume that $\phi$ is a $U$-mixer of $V\times U$. Then $\phi\notin\Sym(V)\times\Sym(U)$.
Considering the Cartesian-prime factorizations of $S(\Gamma)$, $S_+(\Sigma)$ and $S(\Gamma)\Bbox S_+(\Sigma)$, we derive by Lemma~\ref{thm:cartesian} that $S(\Gamma)$ and $S_+(\Sigma)$ can be written as
\[
S(\Gamma)=\Gamma'\Bbox\Delta_1\Bbox\cdots\Bbox\Delta_a\ \text{ and }\ S_+(\Sigma)=\Sigma'\Bbox\Delta_{a+1}\Bbox\cdots\Bbox\Delta_{a+b}
\]
for some graphs $\Gamma'$, $\Sigma'$ and $\Delta_1,\ldots,\Delta_{a+b}$, where $a$ and $b$ are positive integers, such that $\Delta_1,\ldots,\Delta_{a+b}$ are isomorphic Cartesian-prime graphs and $\Delta_{a+b}$ is sent to $\Delta_1$ by $\phi$. Then
\[
V\times U=V(\Gamma')\times V(\Delta_1)\times\cdots\times V(\Delta_a)\times V(\Sigma')\times V(\Delta_{a+1})\times\cdots\times V(\Delta_{a+b}).
\]
Take $x$ and $y$ in $V\times U$ such that
\[
x=(v_0,v_1,\ldots,v_a,u_0,u_1,\ldots,u_b)\ \text{ and }\ y=(v_0,v_1,\ldots,v_a,u_0,u_1,\ldots,u_{b-1},u_b')
\]
with distinct $u_b$ and $u_b'$ in $V(\Delta_{a+b})$. Then $x^{\pi}=y^{\pi}$, and since $\phi$ sends $\Delta_{a+b}$ to $\Delta_1$, we have $x^{\phi\pi}\neq y^{\phi\pi}$. This shows that $x^{\rho\pi}\neq y^{\rho\pi}$, and so $\rho$ is a $V$-mixer.
\end{proof}

\begin{lemma}\label{lem:S(Gamma)}
Let $\Gamma$ and $\Sigma$ be connected twin-free graphs such that $\Gamma$ is stable and $\Sigma$ is bipartite with parts $U$ and $W$ of size at least three, and let $\rho\in\Aut(\Gamma\times\Sigma)$ be a $\Gamma$-mixer that stabilizes $V(\Gamma)\times U$. Suppose that $S(\Gamma)$ is Cartesian-prime. Then $\rho$ induces a $U$-mixer of $V(\Gamma)\times U$ and a $W$-mixer of $V(\Gamma)\times W$.
\end{lemma}

\begin{proof}
Let $V=V(\Gamma)$, let $S_+(\Sigma)$ and $S_-(\Sigma)$ be the induced subgraphs of $S(\Sigma)$ by $U$ and $W$ respectively, and let $\pi$ be the projection of $V\times V(\Sigma)$ to $V$.
By Lemmas~\ref{thm:cartesian} and~\ref{le:lecsb}, $S(\Gamma)\Bbox S_+(\Sigma)$ and $S(\Gamma)\Bbox S_-(\Sigma)$ are the connected components of $S(\Gamma\times\Sigma)=S(\Gamma)\Bbox S(\Sigma)$, and $\rho$ induces automorphisms on $S(\Gamma)\Bbox S_+(\Sigma)$ and $S(\Gamma)\Bbox S_-(\Sigma)$, say $\phi$ and $\psi$, respectively.

First suppose that $\phi\in\Sym(V)\times\Sym(U)$ and $\psi\in\Sym(V)\times\Sym(W)$.
Then for each $v\in V$, $i,i'\in U$ and $j,j'\in W$ we have
\begin{align*}
(v,i)^{\rho\pi}&=(v,i)^{\phi\pi}=(v,i')^{\phi\pi}=(v,i')^{\rho\pi},\\
(v,j)^{\rho\pi}&=(v,j)^{\psi\pi}=(v,j')^{\psi\pi}=(v,j')^{\rho\pi}.
\end{align*}
Since $\rho$ is a $\Gamma$-mixer, there exist $v'\in V$, $i'\in U$ and $j'\in W$ such that $(v',i')^{\rho\pi}\neq(v',j')^{\rho\pi}$. Take an edge $\{i,j\}$ of $\Sigma$ with $i\in U$ and $j\in W$. It follows that $(v',i)^{\rho\pi}\neq(v',j)^{\rho\pi}$.
For each $v\in V$, let $v^\alpha=(v,i)^{\rho\pi}$ and $v^\beta=(v,j)^{\rho\pi}$.
Since $\phi\in\Sym(V)\times\Sym(U)$ and $\psi\in\Sym(V)\times\Sym(W)$, both $\alpha$ and $\beta$ are permutations on $V$.  Observing that
\begin{align*}
(s,t)\in A(\Gamma)&\,\Rightarrow\,((s,i)^\rho,(t,j)^\rho)\in A(\Gamma\times\Sigma)\\
&\,\Rightarrow\,((s,i)^{\rho\pi},(t,j)^{\rho\pi})\in A(\Gamma)\,\Rightarrow\,(s^\alpha,t^\beta)\in A(\Gamma)
\end{align*}
and that $(v')^\alpha=(v',i)^{\rho\pi}\neq(v',j)^{\rho\pi}=(v')^\beta$, we conclude that $(\alpha,\beta)$ is a non-diagonal TF-automorphism of $\Gamma$. Then Lemma~\ref{lem:TF}\,\eqref{lem:TF-a} asserts that $\Gamma$ is unstable, a contradiction.

Next suppose that $\phi\notin\Sym(V)\times\Sym(U)$ and $\psi\in\Sym(V)\times\Sym(W)$.
Recall that $S(\Gamma)$ is Cartesian-prime. Considering the Cartesian-prime factorizations of $S_+(\Sigma)$ and $S(\Gamma)\Bbox S_+(\Sigma)$, we deduce by Lemma~\ref{thm:cartesian} that
\[
S_+(\Sigma)=\Sigma'\Bbox\Delta_1\Bbox\cdots\Bbox\Delta_a
\]
for some graphs $\Sigma'$ and $\Delta_1,\ldots,\Delta_a$ such that $S(\Gamma),\Delta_1,\ldots,\Delta_a$ are isomorphic and cyclically permuted by $\phi$. Thus, for each $i\in U$,
\begin{equation}\label{eq:eqa}
(V\times \{i\})^\phi=\{v\}\times I\,\text{ for some $v\in V$ and $I\subseteq U$}.
\end{equation}
Take any $j\in W$ such that $\{i,j\}\in E(\Sigma)$. For each $t\in V$, since $\psi\in\Sym(V)\times\Sym(W)$, there exists $s\in V$ such that $(s,j)^{\psi\pi}=t$. Then taking any neighbor $r$ of $s$ in $\Gamma$, we derive from~\eqref{eq:eqa} and $\{(r,i),(s,j)\}\in E(\Gamma\times\Sigma)$ that $v=(r,i)^{\phi\pi}$ is adjacent to $(s,j)^{\psi\pi}=t$. This shows that $v$ is adjacent to all the vertices in $\Gamma$. Since $\phi$ is bijective, $v$ runs over $V$ when $i$ runs over $U$. Hence $\Gamma$ is a complete graph with a loop attached to each vertex, contradicting to that $\Gamma$ is twin-free.

Similarly, the case when $\phi\in\Sym(V)\times\Sym(U)$ and $\psi\notin\Sym(V)\times\Sym(W)$ cannot occur. Therefore, $\phi\notin\Sym(V)\times\Sym(U)$ and $\psi\notin\Sym(V)\times\Sym(W)$. Following the argument in the previous paragraph up to~\eqref{eq:eqa}, we see that $\phi$ is a $U$-mixer. Similarly, $\psi$ is a $W$-mixer.
\end{proof}

\begin{lemma}\label{lem:S(Sigma)}
Let $\Gamma$ and $\Sigma$ be connected twin-free graphs such that $\Gamma$ is stable and $\Sigma$ is bipartite with parts $U$ and $W$ of size at least three, and let $\rho\in\Aut(\Gamma\times\Sigma)$ be a $\Gamma$-mixer that stabilizes $V(\Gamma)\times U$. Suppose that each connected component of $S(\Sigma)$ is Cartesian-prime and no connected component of $B(\Sigma)$ is complete. Then $\rho$ induces a $U$-mixer of $V(\Gamma)\times U$ and a $W$-mixer of $V(\Gamma)\times W$.
\end{lemma}

\begin{proof}
Let $\pi_\Gamma$ and $\pi_\Sigma$ be the projections of $V(\Gamma)\times V(\Sigma)$ to $V(\Gamma)$ and $V(\Sigma)$ respectively, let $S_+(\Sigma)$ and $S_-(\Sigma)$ be the induced subgraphs of $S(\Sigma)$ by $U$ and $W$ respectively, and let $V=V(\Gamma)$.
By Lemmas~\ref{thm:cartesian} and~\ref{le:lecsb}, $S(\Gamma)\Bbox S_+(\Sigma)$ and $S(\Gamma)\Bbox S_-(\Sigma)$ are the connected components of $S(\Gamma\times\Sigma)=S(\Gamma)\Bbox S(\Sigma)$, and $\rho$ induces automorphisms on $S(\Gamma)\Bbox S_+(\Sigma)$ and $S(\Gamma)\Bbox S_-(\Sigma)$, say $\phi$ and $\psi$, respectively. By the second paragraph in the proof of Lemma~\ref{lem:S(Gamma)}, the case when $\phi\in\Sym(V)\times\Sym(U)$ and $\psi\in\Sym(V)\times\Sym(W)$ is not possible.

Suppose that $\phi\notin\Sym(V)\times\Sym(U)$ and $\psi\in\Sym(V)\times\Sym(W)$. Recall that $S_+(\Sigma)$ is Cartesian-prime.
Considering the Cartesian-prime factorizations of $S(\Gamma)$ and $S(\Gamma)\Bbox S_+(\Sigma)$, we infer by Lemma~\ref{thm:cartesian} that
\[
S(\Gamma)=\Gamma'\Bbox\Delta_1\Bbox\cdots\Bbox\Delta_a
\]
for some graphs $\Gamma'$ and $\Delta_1,\ldots,\Delta_a$ such that $\Delta_1,\ldots,\Delta_a,S_+(\Sigma)$ are isomorphic and cyclically permuted by $\phi$. Consequently,
\[
(v,i)^{\phi\pi_\Sigma}=(v,j)^{\phi\pi_\Sigma}
\]
for each $v\in V$ and $i,j\in U$. Since $|U|\geq 3$ and the connected component of $B(\Sigma)$ with vertex set $U$ is not complete, $B(\Sigma)$ has a $2$-path $(i_1,i_2,i_3)$ in $U$ with $i_1$ and $i_3$ not adjacent. Take a common neighbor $j_1$ of $i_1$ and $i_2$ in $\Sigma$, a common neighbor $j_2$ of $i_2$ and $i_3$ in $\Sigma$, and an edge $\{s,t\}$ of $\Gamma$. Then $j_1,j_2\in W$, and since $\psi\in\Sym(V)\times\Sym(W)$,
\begin{equation}\label{eq:(t,j_2)}
(t,j_2)^\rho=(t,j_2)^\psi=((t,j_2)^{\psi\pi_\Gamma},(t,j_2)^{\psi\pi_\Sigma})=((t,j_1)^{\psi\pi_\Gamma},(t,j_2)^{\psi\pi_\Sigma})\in V\times W.
\end{equation}
Since $\{(s,i_1),(t,j_1)\}$ and $\{(s,i_2),(t,j_2)\}$ are edges of $\Gamma\times\Sigma$, it follows that $(s,i_1)^{\rho\pi_\Gamma}$ is adjacent to $(t,j_1)^{\rho\pi_\Gamma}=(t,j_1)^{\psi\pi_\Gamma}$ in $\Gamma$, and that $(s,i_2)^{\rho\pi_\Sigma}$ is adjacent to $(t,j_2)^{\rho\pi_\Sigma}=(t,j_2)^{\psi\pi_\Sigma}$ in $\Sigma$. Combining the latter with
\[
(s,i_1)^{\rho\pi_\Sigma}=(s,i_1)^{\phi\pi_\Sigma}=(s,i_2)^{\phi\pi_\Sigma}=(s,i_2)^{\rho\pi_\Sigma},
\]
we deduce that $((s,i_1)^{\rho\pi_\Gamma},(s,i_1)^{\rho\pi_\Sigma})$ is adjacent to $((t,j_1)^{\psi\pi_\Gamma},(t,j_2)^{\psi\pi_\Sigma})$ in $\Gamma\times\Sigma$. In view of~\eqref{eq:(t,j_2)}, we conclude that $(s,i_1)^\rho$ is adjacent to $(t,j_2)^\rho$ in $\Gamma\times\Sigma$. As a consequence, $(s,i_1)$ is adjacent to $(t,j_2)$ in $\Gamma\times\Sigma$, and so $i_1$ is adjacent to $j_2$ in $\Sigma$. However, this implies that $i_1$ is adjacent to $i_3$ in $B(\Sigma)$, a contradiction.

Similarly, the case when $\phi\in\Sym(V)\times\Sym(U)$ and $\psi\notin\Sym(V)\times\Sym(W)$ cannot occur. Therefore, $\phi\notin\Sym(V)\times\Sym(U)$ and $\psi\notin\Sym(V)\times\Sym(W)$. Considering the Cartesian-prime factorizations of $S(\Gamma)$ and $S(\Gamma)\Bbox S_+(\Sigma)$, we deduce by Lemma~\ref{thm:cartesian} that $S(\Gamma)=\Gamma'\Bbox\Delta_1\Bbox\cdots\Bbox\Delta_a$
for some graphs $\Gamma'$ and $\Delta_1,\ldots,\Delta_a$ such that $\Delta_1,\ldots,\Delta_a,S_+(\Sigma)$ are isomorphic and cyclically permuted by $\phi$. Accordingly, each vertex in $V\times U$ has the form
\[
(v_0,v_1,\ldots,v_a,u)\in V(\Gamma')\times V(\Delta_1)\times\cdots\times V(\Delta_a)\times U.
\]
Take $x$ and $y$ in $V\times U$ such that
\[
x=(v_0,v_1,\ldots,v_{a-1},v_a,u)\ \text{ and }\ y=(v_0,v_1,\ldots,v_{a-1},v_a',u)
\]
with distinct $v_a$ and $v_a'$ in $V(\Delta_{a})$. Then $x^{\phi\pi_\Gamma}=y^{\phi\pi_\Gamma}$ and $x^{\phi\pi_\Sigma}\neq y^{\phi\pi_\Sigma}$, which shows that $\phi$ is a $U$-mixer. In the same vein we see that $\psi$ is a $W$-mixer.
\end{proof}

Combining the above three lemmas, we are now in a position to establish the following:

\begin{proposition}\label{lem:mixer}
Let $\Gamma$ and $\Sigma$ be connected twin-free graphs such that $\Gamma$ is stable and $\Sigma$ is bipartite with parts $U$ and $W$ of size at least three, such that either $S(\Gamma)$ is Cartesian-prime, or each connected component of $S(\Sigma)$ is Cartesian-prime and no connected component of $B(\Sigma)$ is complete. Then for each automorphism $\rho$ of $\Gamma\times\Sigma$ that stabilizes $V(\Gamma)\times U$, the following are equivalent:
\begin{enumerate}[\rm(a)]
\item\label{lem-mixer-a} $\rho$ is a $\Gamma$-mixer of $V(\Gamma)\times V(\Sigma)$;
\item\label{lem-mixer-b} $\rho$ is a $\Sigma$-mixer of $V(\Gamma)\times V(\Sigma)$;
\item\label{lem-mixer-c} $\rho$ induces a $U$-mixer of $V(\Gamma)\times U$;
\item\label{lem-mixer-d} $\rho$ induces a $W$-mixer of $V(\Gamma)\times W$.
\end{enumerate}
\end{proposition}

\begin{proof}
According to Lemma~\ref{le:lecsb}, $\rho$ is an automorphism of $S(\Gamma\times\Sigma)=S(\Gamma)\Bbox S(\Sigma)$. Hence Lemma~\ref{lem:Gamma-mixer} implies~\eqref{lem-mixer-b}\,$\Rightarrow$\,\eqref{lem-mixer-a}. Moreover, Lemmas~\ref{lem:S(Gamma)} and~\ref{lem:S(Sigma)} together imply~\eqref{lem-mixer-a}\,$\Rightarrow$\,\eqref{lem-mixer-c} and~\eqref{lem-mixer-a}\,$\Rightarrow$\,\eqref{lem-mixer-d}. Noting by definition that~\eqref{lem-mixer-c}\,$\Rightarrow$\,\eqref{lem-mixer-b} and~\eqref{lem-mixer-d}\,$\Rightarrow$\,\eqref{lem-mixer-b}, we conclude that~\eqref{lem-mixer-a}--\eqref{lem-mixer-d} are all equivalent.
\end{proof}

\section{Proof of Theorem~\ref{thm3}}\label{sec2}

\subsection{Proof of Theorem~\ref{thm3}\,\eqref{thm3a}}

We prove the contrapositive of the statement. Suppose that $\Gamma$ is unstable. Since $(\Gamma,\Sigma)$ is a nontrivial graph pair with $\Sigma$ bipartite, $\Gamma$ is connected and non-bipartite. Then by Lemma~\ref{lem:TF}, there exists a nontrivial TF-automorphism $(\alpha,\beta)$ of $\Gamma$. In other words, $\alpha$ and $\beta$ are distinct permutations on $V(\Gamma)$ such that~\eqref{eq:TF} holds for all $s,t\in V(\Gamma)$.
Let $V(\Sigma)=\{v_1,\cdots,v_n\}$, and for $i\in\{1,\dots,n\}$, let $\alpha_i=\alpha$ if $v_i\in U$ and $\alpha_i=\beta$ if $v_i\in W$. Now consider $(\alpha_1,\dots,\alpha_n)\in\Sym(V(\Gamma))^n$. Since $\{U,W\}$ is a bipartition of $\Sigma$, each edge of $\Sigma$ has the form $\{v_i,v_j\}$ with $v_i\in U$ and $v_j\in W$. For this edge $\{v_i,v_j\}$, by~\eqref{eq:TF}, we have for $s,t\in V(\Gamma)$ that
\[
(s,t)\in A(\Gamma)\,\Leftrightarrow\,(s^\alpha,t^\beta)\in A(\Gamma)\,\Leftrightarrow\,(s^{\alpha_i},t^{\alpha_j})\in A(\Gamma).
\]
Moreover, since $\alpha\neq\beta$, the permutations $\alpha_1,\ldots,\alpha_n$ are not all equal. Hence $(\Gamma, \Sigma)$ is unstable by~\cite[Lemma~2.6(a)]{QXZZ2021}.

\subsection{Proof of Theorem~\ref{thm3}\,\eqref{thm3b}}

Let $\Gamma$ be stable, and denote $V=V(\Gamma)$, $X=V\times U$ and $Y=V\times W$. Define a mapping
\[
f\colon\,\Aut(\Gamma)\times\Aut(\Gamma)\times\TFA(\Sigma)\to\TFA(\Gamma\times\Sigma),\ \ (\sigma,\tau,(\alpha_0,\beta_0))\mapsto(\alpha,\beta)
\]
by letting $\alpha|_X=(\sigma,\alpha_0)|_X$, $\alpha|_Y=(\tau,\alpha_0)|_Y$, $\beta|_X=(\tau,\beta_0)|_X$ and  $\beta|_Y=(\sigma,\beta_0)|_Y$. It is straightforward to verify that $f$ indeed maps $\Aut(\Gamma)\times\Aut(\Gamma)\times\TFA(\Sigma)$ into $\TFA(\Gamma\times\Sigma)$, and that $f$ is injective. For each $(\alpha,\beta)\in\TFA(\Gamma\times\Sigma)$, Lemma~\ref{lem:TFA} asserts the existence of TF-automorphisms $(\alpha_+,\beta_-)$ and $(\alpha_-,\beta_+)$ of $\Gamma$ and TF-automorphism $(\alpha_0,\beta_0)$ of $\Sigma$ with $\alpha|_X=(\alpha_+,\alpha_0)|_X$, $\alpha|_Y=(\alpha_-,\alpha_0)|_Y$, $\beta|_X=(\beta_+,\beta_0)|_X$ and $\beta|_Y=(\beta_-,\beta_0)|_Y$. Since $\Gamma$ is stable, Lemma~\ref{lem:TF}\,\eqref{lem:TF-a} implies that every TF-automorphism of $\Gamma$ is diagonal. Therefore, $\alpha_+=\beta_-$ and $\alpha_-=\beta_+$ are both automorphisms of $\Gamma$. It follows that
\[
\alpha|_X=(\alpha_+,\alpha_0)|_X,\ \ \alpha|_Y=(\beta_+,\alpha_0)|_Y,\ \ \beta|_X=(\beta_+,\beta_0)|_X,\ \ \beta|_Y=(\alpha_+,\beta_0)|_Y,
\]
which means $(\alpha,\beta)=(\alpha_+,\beta_+,(\alpha_0,\beta_0))^f$. This shows that $f$ is surjective, and so $f$ is a bijection. As a consequence,
\begin{equation}\label{eq:|TFA|}
|\Aut(\Gamma)|^2|\TFA(\Sigma)|=|\TFA(\Gamma\times\Sigma)|.
\end{equation}
Note that $\Gamma\times\Sigma$ is connected and bipartite with bipartition $\{X,Y\}$, and that, as there is an automorphism of $\Sigma$ swapping $U$ and $W$, there is an automorphism of $\Gamma\times\Sigma$ swapping $X$ and $Y$. We obtain by Lemma~\ref{lem:TF-bipartite} that
\[
|\Aut(\Sigma\times K_2)|=4|\TFA(\Sigma)|\ \text{ and }\ |\Aut(\Gamma\times\Sigma\times K_2)|=4|\TFA(\Gamma\times\Sigma)|.
\]
This in conjunction with~\eqref{eq:|TFA|} yields
\begin{equation}\label{eq:|Aut|}
|\Aut(\Gamma)|^2|\Aut(\Sigma\times K_2)|=|\Aut(\Gamma\times\Sigma\times K_2)|.
\end{equation}
Since both $\Sigma$ and $\Gamma\times\Sigma$ are connected and bipartite, we have
\[
\Aut(\Sigma\times K_2)=\Aut(\Sigma)\wr Z_2\ \text{ and }\ \Aut(\Gamma\times\Sigma\times K_2)=\Aut(\Gamma\times\Sigma)\wr Z_2.
\]
Hence~\eqref{eq:|Aut|} turns out to be
\[
2|\Aut(\Gamma)|^2|\Aut(\Sigma)|^2=2|\Aut(\Gamma\times\Sigma)|^2,
\]
which leads to $|\Aut(\Gamma)||\Aut(\Sigma)|=|\Aut(\Gamma\times\Sigma)|$. Thus, $(\Gamma,\Sigma)$ is stable.

\subsection{Proof of Theorem~\ref{thm3}\,\eqref{thm3c}}

Since $(\Gamma,\Sigma)$ is a nontrivial pair with $\Sigma$ bipartite, $\Gamma$ is connected, non-bipartite and twin-free, and $\Sigma$ is connected. Then since $\{U,W\}$ is a bipartition of $\Sigma$, Lemma~\ref{le:lecsb} shows that $S(\Sigma)$ has precisely two connected components with vertex sets $U$ and $W$. Let $S_+(\Sigma)$ and $S_-(\Sigma)$ be the subgraphs of $S(\Sigma)$ induced by $U$ and $W$ respectively.

Let $V=V(\Gamma)$, let $\Pi=\Gamma\times\Sigma$, and let $\pi_\Gamma$ and $\pi_\Sigma$ be the projections of $V\times V(\Sigma)$ to $V$ and $V(\Sigma)$ respectively. Then $\Pi$ is bipartite with bipartition $\{X,Y\}$, where
\[
X:=V\times U\ \text{ and }\ Y:=V\times W.
\]
Accordingly, $S(\Pi)$ has precisely two connected components, denoted by $S_+(\Pi)$ and $S_-(\Pi)$, with vertex sets $X$ and $Y$, respectively. By Lemma~\ref{le:lecsb},
\begin{equation}\label{eq:S+S-}
S(\Pi)=S(\Gamma)\Bbox S(\Sigma),\ \ S_+(\Pi)=S(\Gamma)\Bbox S_+(\Sigma)\ \text{ and }\ S_-(\Pi)=S(\Gamma)\Bbox S_-(\Sigma).
\end{equation}
Suppose for a contradiction that $\Gamma$ is stable while $(\Gamma,\Sigma)$ is unstable. Then Lemma~\ref{le:legsm} and Lemma~\ref{le:lecsb}\,\eqref{le:leas} together imply the existence of some automorphism of $S(\Pi)$ that is either a $\Gamma$-mixer or a $\Sigma$-mixer. Let $\tau$ be such an automorphism.

By Lemma~\ref{claim1}, there exists $\sigma\in\Aut(S(\Sigma))$ with $|\sigma|\leq2$ such that $\rho:=(\mathrm{id},\sigma)\tau\in\Aut(S(\Pi))$ stabilizes $X$, where $\id$ is the identity permutation on $V$, and that if $|\sigma|=2$ then $\sigma$ interchanges $U$ and $W$.
It follows that $\rho$ induces some automorphisms, say $\phi$ and $\psi$, on $S_+(\Pi)$ and $S_-(\Pi)$ respectively.
Since $(\mathrm{id},\sigma)$ preserves both the $\Gamma$-partition and the $\Sigma$-partition, $\rho$ is either a $\Gamma$-mixer or a $\Sigma$-mixer.
Recall that each connected component of $S(\Sigma)$ is Cartesian-prime and no connected component of $B(\Sigma)$ is complete.
Then we derive from Proposition~\ref{lem:mixer} that $\phi$ is a $U$-mixer and $\psi$ is a $W$-mixer. Thus, viewing~\eqref{eq:S+S-}, we conclude by Lemma~\ref{thm:cartesian} that
\begin{equation}\label{eq:S+S-decom}
S(\Gamma)=\Gamma_+\Box S_+(\Sigma)^{\Box,c-1}=\Gamma_-\Box S_-(\Sigma)^{\Box,d-1}
\end{equation}
for some integers $c,d\geq 2$ and graphs $\Gamma_+,\Gamma_-$ Cartesian-coprime to $S_+(\Sigma),S_-(\Sigma)$ respectively, and that
\begin{align*}
\phi&=(\phi_0,(\phi_1,\ldots,\phi_c,g^{-1}))\in\Aut(\Gamma_+)\times\Aut(S_+(\Sigma)^{\Box,c})=\Aut(\Gamma_+\Box S_+(\Sigma)^{\Box,c}),\\ \psi&=(\psi_0,(\psi_1,\ldots,\psi_d,h^{-1}))\in\Aut(\Gamma_-)\times\Aut(S_-(\Sigma)^{\Box,d})=\Aut(\Gamma_-\Box S_-(\Sigma)^{\Box,d})
\end{align*}
with $(\phi_1,\ldots,\phi_c,g^{-1})\in\Aut(S_+(\Sigma))\wr S_c$ and $(\psi_1,\ldots,\psi_d,h^{-1})\in\Aut(S_-(\Sigma))\wr S_d$ satisfying the following conditions:
\begin{enumerate}[(i)]
\item\label{condition1} $(\phi_1,\ldots,\phi_c)\in\Aut(S_+(\Sigma))^c$ and $g\in S_c$ such that $c^g\in\{1,\ldots,c-1\}$;
\item\label{condition2} each $(x_0,x_1,\ldots,x_c)\in X=V(S_+(\Pi))$ is mapped by $\phi$ to $(x_0^{\phi_{0}},x_{1^g}^{\phi_{1^g}},\ldots,x_{c^g}^{\phi_{c^g}})$;
\item\label{condition3} $(\psi_1,\ldots,\psi_d)\in\Aut(S_-(\Sigma))^d$ and $h\in S_d$ such that $d^h\in\{1,\ldots,d-1\}$;
\item\label{condition4} each $(y_0,y_1,\ldots,y_d)\in Y=V(S_-(\Pi))$ is mapped by $\psi$ to $(y_0^{\psi_{0}},y_{1^h}^{\psi_{1^h}},\ldots,y_{d^h}^{\psi_{d^h}})$.
\end{enumerate}
Let $a=c^g\in\{1,\ldots,c-1\}$ and $b=d^h\in\{1,\ldots,d-1\}$, and fix some $\tilde{u}\in U$ and $\tilde{w}\in W$ such that $\tilde{u}$ is adjacent to $\tilde{w}$ in $\Sigma$.

\begin{claim}\label{claim5}
$(v,(v,\tilde{u})^{\tau\pi_\Sigma})^{\tau\pi_\Gamma}=(v,(v,\tilde{w})^{\tau\pi_\Sigma})^{\tau\pi_\Gamma}$ for each $v\in V$.
\end{claim}

\begin{proof}
For each $v\in V$, let $v^\alpha=(v,(v,\tilde{u})^{\tau\pi_\Sigma})^{\tau\pi_\Gamma}\in V$ and $v^\beta=(v,(v,\tilde{w})^{\tau\pi_\Sigma})^{\tau\pi_\Gamma}\in V$.
We first show that $\alpha$ and $\beta$ are permutations on $V$. Let $s$ and $t$ be elements of $V$ with $s^\alpha=t^\alpha$. If $|\sigma|=1$, then writing
\begin{align*}
s&=(s_0,s_1,\ldots,s_{c-1})\in V(\Gamma_+)\times U\times\cdots\times U,\\
t&=(t_0,t_1,\ldots,t_{c-1})\in V(\Gamma_+)\times U\times\cdots\times U,
\end{align*}
we deduce from condition~(ii) that $s^\alpha=(s,s_a^{\phi_a})^{\phi\pi_\Gamma}$ and $t^\alpha=(t,t_a^{\phi_a})^{\phi\pi_\Gamma}$.
Moreover, by~(ii),
\[
(s,s_a^{\phi_a})^{\phi\pi_\Gamma}=(t,t_a^{\phi_a})^{\phi\pi_\Gamma}\,\Rightarrow\,(s_a^{\phi_a})^{\phi_c}=(t_a^{\phi_a})^{\phi_c}
\,\Rightarrow\,s_a^{\phi_a}=t_a^{\phi_a},
\]
and so $(s,s_a^{\phi_a})^{\phi\pi_\Gamma}=(t,t_a^{\phi_a})^{\phi\pi_\Gamma}\,\Rightarrow\,(s,\tilde{u})^\phi=(t,\tilde{u})^\phi\,\Rightarrow\,s=t$.
If $|\sigma|=2$, then writing
\begin{align*}
s&=(s_0,s_1,\ldots,s_{d-1})\in V(\Gamma_-)\times W\times\cdots\times W,\\
t&=(t_0,t_1,\ldots,t_{d-1})\in V(\Gamma_-)\times W\times\cdots\times W,
\end{align*}
we deduce from~(iv) that $s^\alpha=(s,s_b^{\psi_b})^{\psi\pi_\Gamma}$ and $t^\alpha=(t,t_b^{\psi_b})^{\psi\pi_\Gamma}$.
Similarly as above, we have
\[
(s,s_b^{\psi_b})^{\psi\pi_\Gamma}=(t,t_b^{\psi_b})^{\psi\pi_\Gamma}\,\Rightarrow\,s=t.
\]
Thus, $\alpha$ is a permutation on $V$, and so is $\beta$ for the same reason.

Let $(s,t)\in A(\Gamma)$. Then $((s,\tilde{u})^\tau,(t,\tilde{w})^\tau)\in A(\Pi)$, and so $((s,\tilde{u})^{\tau\pi_\Sigma},(t,\tilde{w})^{\tau\pi_\Sigma})\in A(\Sigma)$, which implies that $((s,(s,\tilde{u})^{\tau\pi_\Sigma}),(t,(t,\tilde{w})^{\tau\pi_\Sigma}))\in A(\Pi)$. This in turn implies that
\[
(s^\alpha,t^\beta)=\big((s,(s,\tilde{u})^{\tau\pi_\Sigma})^{\tau\pi_\Gamma},(t,(t,\tilde{w})^{\tau\pi_\Sigma})^{\tau\pi_\Gamma}\big)\in A(\Gamma).
\]
However, $\Gamma$ is stable and hence admits no TF-automorphisms. Hence $\alpha=\beta$, as the claim asserts.
\end{proof}

Recall that $S(\Sigma)$ admits an automorphism $\delta$ that fixes some vertex, say $u$, in $U$ but not any vertex in $W$.
Since $U$ and $W$ are the vertex sets of the two connected components of $S(\Sigma)$ and since $\delta$ fixes $u\in U$, it follows that $\delta$ stabilizes each of $U$ and $W$. Let $\tau'=\tau(\mathrm{id},\delta)$. Then $\tau'$ is also an automorphism of $\Pi$ that is either a $\Gamma$-mixer or a $\Sigma$-mixer, and we derive from Claim~\ref{claim5} that
\begin{equation}\label{eq:TF1}
(v,(v,\tilde{u})^{\tau\pi_\Sigma\delta})^{\tau\pi_\Gamma}=(v,(v,\tilde{u})^{\tau'\pi_\Sigma})^{\tau'\pi_\Gamma}
=(v,(v,\tilde{w})^{\tau'\pi_\Sigma})^{\tau'\pi_\Gamma}=(v,(v,\tilde{w})^{\tau\pi_\Sigma\delta})^{\tau\pi_\Gamma}
\end{equation}
for each $v\in V$. If $|\sigma|=1$, then take
\[
v=(v_0,v_1,\ldots,v_{c-1})\in V(\Gamma_+)\times U\times\cdots\times U
\]
such that $(v,\tilde{u})^{\tau\pi_\Sigma}=v_a^{\phi_a}=u$. In this case, $(v,\tilde{u})^{\tau\pi_\Sigma\delta}=(v,\tilde{u})^{\tau\pi_\Sigma}$, and so
\begin{equation}\label{eq:TF2}
(v,(v,\tilde{u})^{\tau\pi_\Sigma\delta})^{\tau\pi_\Gamma}=(v,(v,\tilde{u})^{\tau\pi_\Sigma})^{\tau\pi_\Gamma},
\end{equation}
but as $(v,\tilde{w})^{\tau\pi_\Sigma}\in W$, we have $(v,\tilde{w})^{\tau\pi_\Sigma\delta}\neq(v,\tilde{w})^{\tau\pi_\Sigma}$ and hence by~(iv),
\begin{equation}\label{eq:TF3}
(v,(v,\tilde{w})^{\tau\pi_\Sigma\delta})^{\tau\pi_\Gamma}=(v,(v,\tilde{w})^{\tau\pi_\Sigma\delta})^{\psi\pi_\Gamma}
\neq(v,(v,\tilde{w})^{\tau\pi_\Sigma})^{\psi\pi_\Gamma}=(v,(v,\tilde{w})^{\tau\pi_\Sigma})^{\tau\pi_\Gamma}.
\end{equation}
Combining~\eqref{eq:TF1} with~\eqref{eq:TF2} and~\eqref{eq:TF3}, we obtain $(v,(v,\tilde{u})^{\tau\pi_\Sigma})^{\tau\pi_\Gamma}\neq(v,(v,\tilde{w})^{\tau\pi_\Sigma})^{\tau\pi_\Gamma}$, which contradicts Claim~\ref{claim5}.
If $|\sigma|=2$, then take
\[
v\in V(\Gamma_-)\times W\times\cdots\times W
\]
such that $(v,\tilde{w})^{\tau\pi_\Sigma}=u$, and we derive a contradiction in the same vein.

\section{Stability of graph pairs involving a cycle}\label{sec3}

\subsection{The gap in the proof of~\cite[Theorem~1.8]{GLX2025}}\label{sec4}

There is a gap in the published proof of~\cite[Theorem~1.8]{GLX2025} (namely, our Theorem~\ref{thm1}), which originates in the proof of~\cite[Lemma~5.4]{GLX2025}. Specifically, the inequality
\begin{equation}\label{eq:wrong}
|N_\Delta((y,k))\cap N_\Delta((z,\ell))|\leq|N_\Gamma(y)\cap N_\Gamma(z)|
\end{equation}
there (the eighth line on Page~162) is only valid when $k\neq\ell$, and may fail if $k=\ell$. In fact, by definition of direct product graph $\Delta=\Gamma\times C_m$,
\[
|N_\Delta((y,k))\cap N_\Delta((z,\ell))|=|N_\Gamma(y)\cap N_\Gamma(z)||N_{C_m}(k)\cap N_{C_m}(\ell)|,
\]
and when $k=\ell$, the left-hand side of~\eqref{eq:wrong} becomes $2|N_\Gamma(y)\cap N_\Gamma(z)|$, not satisfying~\eqref{eq:wrong} for $y$ and $z$ with common neighbors in $\Gamma$.

\subsection{Proof of Theorem~\ref{thm1}}

Let $n\geq3$ be odd, let $\Gamma$ be a connected twin-free bipartite graph with bipartition $\{U,W\}$, and let
\[
\Pi=\Gamma\times C_n.
\]
We use $\pi_{\Gamma}$ and $\pi_C$ to denote the projections of $V(\Pi)=V(\Gamma)\times V(C_n)$ to $V(\Gamma)$ and $V(C_n)$, respectively. It is easy to know (see for instance~\cite[Theorem~5.9]{HIK2011}) that $\Pi$ is connected and bipartite with a bipartition $\{X,Y\}$, where
\[
X:=U\times V(C_n)\ \text{ and }\ Y:=W\times V(C_n).
\]
By Lemma~\ref{le:lecsb}, the Cartesian skeleton $S(\Gamma)$ of $\Gamma$ has precisely two connected components with vertex sets $U$ and $W$ respectively, and
\[
S(\Pi)=S(\Gamma)\Bbox S(C_n)
\]
has precisely two connected components with vertex sets $X$ and $Y$ respectively.
Let $S_+(\Gamma)$ and $S_-(\Gamma)$ be the subgraphs of $S(\Gamma)$ induced by $U$ and $W$ respectively, and let $S_+(\Pi)$ and $S_-(\Pi)$ be the subgraphs of $S(\Pi)$ induced by $X$ and $Y$ respectively. Then
\begin{equation}\label{eq:S0S1}
S_+(\Gamma)=S_+(\Gamma)\Bbox S(C_n)\ \text{ and }\ S_-(\Gamma)=S_-(\Gamma)\Bbox S(C_n).
\end{equation}

To prove Theorem~\ref{thm1}, suppose for a contradiction that $(\Gamma,C_n)$ is unstable with $\Gamma$ and $C_n$ coprime. By Lemmas~\ref{le:legsm}, there exists $\tau\in\Aut(\Pi)$ that is either a $\Gamma$-mixer or a $C_n$-mixer.
By Lemma~\ref{claim1}, there exists $\gamma\in\Aut(S(\Gamma))$ with $|\gamma|\leq2$ such that $\tau(\gamma,\mathrm{id})\in\Aut(S(\Pi))$ stabilizes $X$, where $\id$ is the identity permutation on $V(C_n)$. Let
\[
\rho=\tau(\gamma,\mathrm{id}),\ \ \phi=\rho|_X\in\aut(S_+(\Pi))\ \text{ and }\ \psi=\rho|_Y\in\aut(S_-(\Pi)).
\]
Since $(\gamma,\mathrm{id})$ preserves both the $\Gamma$-partition and the $C_n$-partition, it follows that $\rho=\tau(\gamma,\mathrm{id})$ is either a $\Gamma$-mixer or a $C_n$-mixer.
Note from~\eqref{eq:S(Cn)} that $S(C_n)=B(C_n)=C_n$, and hence $S(C_n)$ is Cartesian-coprime. Note also that $C_n$ is stable as
\[
\Aut(C_n\times K_2)=\Aut(C_{2n})=D_{4n}=D_{2n}\times Z_2=\Aut(C_n)\times Z_2.
\]
Then by Proposition~\ref{lem:mixer}, the induced automorphism $\phi$ of $\rho$ on $S_+(\Pi)$ is a $U$-mixer, and the induced automorphism $\psi$ of $\rho$ on $S_-(\Pi)$ is a $W$-mixer. Thus, in view of~\eqref{eq:S0S1}, we conclude by Lemma~\ref{thm:cartesian} that
\begin{equation}\label{eq:S0S1decom}
S_+(\Gamma)=\Gamma_+\Box C_{n}^{\Box,c-1}\ \text{ and }\ S_-(\Gamma)=\Gamma_-\Box C_{n}^{\Box,d-1}
\end{equation}
for some integers $c,d\geq 2$ and graphs $\Gamma_+,\Gamma_-$ that are Cartesian-coprime to $C_n$, and that
\begin{align*}
\phi&=(\phi_0,(\phi_1,\ldots,\phi_c,g^{-1}))\in\Aut(\Gamma_+)\times\Aut(C_{n}^{\Box,c})=\Aut(\Gamma_+\Box C_{n}^{\Box,c})=\Aut(S_+(\Pi)),\\ \psi&=(\psi_0,(\psi_1,\ldots,\psi_d,h^{-1}))\in\Aut(\Gamma_-)\times\Aut(C_{n}^{\Box,d})=\Aut(\Gamma_-\Box C_{n}^{\Box,d})=\Aut(S_-(\Pi))
\end{align*}
with $(\phi_1,\ldots,\phi_c,g^{-1})\in\Aut(C_n)\wr S_c$ and $(\psi_1,\ldots,\psi_d,h^{-1})\in\Aut(C_n)\wr S_d$ satisfying the following conditions:
\begin{enumerate}[(i)]
\item\label{condition1} $(\phi_1,\ldots,\phi_c)\in\Aut(C_n)^c$ and $g\in S_c$ such that $c^g\in\{1,\ldots,c-1\}$;
\item\label{condition2} each $(x_0,x_1,\ldots,x_c)\in X=V(S_+(\Pi))$ is mapped by $\phi$ to $(x_0^{\phi_{0}},x_{1^g}^{\phi_{1^g}},\ldots,x_{c^g}^{\phi_{c^g}})$;
\item\label{condition3} $(\psi_1,\ldots,\psi_d)\in\Aut(C_n)^d$ and $h\in S_d$ such that $d^h\in\{1,\ldots,d-1\}$;
\item\label{condition4} each $(y_0,y_1,\ldots,y_d)\in Y=V(S_-(\Pi))$ is mapped by $\psi$ to $(y_0^{\psi_{0}},y_{1^h}^{\psi_{1^h}},\ldots,y_{d^h}^{\psi_{d^h}})$.
\end{enumerate}
Let $a=c^g\in\{1,\ldots,c-1\}$ and $b=d^h\in\{1,\ldots,d-1\}$. Note from~\eqref{eq:S0S1decom} that each $u\in U$ has the form
\[
u=(u_0,u_1,\ldots,u_{c-1})\in V(\Gamma_+)\times V(C_n)^{c-1}=U
\]
and each $w\in W$ has the form
\[
w=(w_0,w_1,\ldots,w_{d-1})\in V(\Gamma_-)\times V(C_n)^{d-1}=W.
\]
Then conditions~\eqref{condition2} and~\eqref{condition4} yield
\begin{align}
\label{eq:u_a}u_a^{\phi_a}&=u_{c^g}^{\phi_{c^g}}=(u,0)^{\phi\pi_C}=(u,0)^{\rho\pi_C}=(u,0)^{\tau\pi_C},\\
\label{eq:w_b}w_b^{\psi_b}&=w_{d^h}^{\psi_{d^h}}=(w,1)^{\psi\pi_C}=(w,1)^{\rho\pi_C}=(w,1)^{\tau\pi_C}.
\end{align}

For $i\in\{1,\ldots,n\}$, let
\[
L_i=\{u\in U:u_a=i\}\cup\{w\in W:w_b^{\psi_b}=i^{\phi_a}\},
\]
and for $u\in U$ and $w\in W$, let
\begin{align*}
P_u&=\{u_0\}\times\{u_1\}\times\cdots\times\{u_{a-1}\}\times V(C_n)\times \{u_{a+1}\}\times\cdots\times\{u_{c-1}\}\subseteq U,\\
Q_w&=\{w_0\}\times\{w_1\}\times\cdots\times\{w_{b-1}\}\times V(C_n)\times \{w_{b+1}\}\times\cdots\times\{w_{d-1}\}\subseteq W.
\end{align*}
Clearly, $L_1,\ldots,L_n$ are pairwise disjoint. Let $\mathcal{P}=\{P_u:u\in U\}\cup \{Q_w:w\in W\}$. Then $\mathcal{P}$ is a partition of $U\cup W=V(\Gamma)$ (note that different choices of $u$ or $w$ may give the same set $P_u$ or $Q_w$ respectively). The definition of $L_i$, $P_u$ and $Q_w$ implies that~\eqref{it:itdca} of Lemma~\ref{le:ledc} holds. Then we conclude from the following three claims that $L_1,\ldots,L_n$ and $\mathcal{P}$ satisfy the conditions in Lemma~\ref{le:ledc}, and so $\Gamma$ is not coprime to $C_n$, which is a contradiction to complete the proof.

\begin{claim}\label{claim2}
$\{L_1,\ldots,L_n\}$ is a partition of $V(\Gamma)$ into independents sets in $\Gamma$.
\end{claim}

\begin{proof}
To see that $\{L_1,\ldots,L_n\}$ is a partition of $V(\Gamma)$, we only need to notice that
\begin{align*}
\bigcup_{i\in V(C_n)}L_i&=\bigcup_{i\in V(C_n)}\left(\{u\in U:u_a=i\}\cup\{w\in W:w_b^{\psi_b}=i^{\phi_a}\} \right)\\
&=\Big(\bigcup_{i\in V(C_n)}\{u\in U:u_a=i\}\Big)\cup\Big(\bigcup_{i\in V(C_n)}\{w\in W:w_b^{\psi_b}=i^{\phi_a}\}\Big)\\
&=\Big(\bigcup_{i\in V(C_n)}\{u\in U:u_a=i\}\Big)\cup\Big(\bigcup_{j\in V(C_n)}\{w\in W:w_b^{\psi_b}=j\}\Big)\\
&=\,U\cup W.
\end{align*}
Suppose that $L_i$ is not an independent set for some $i\in\{1,\ldots,n\}$. Then, since $\Gamma$ is bipartite with parts $U$ and $W$, there exist $u\in W$ with $u_a=i$ and $w\in W$ with $w_b^{\psi_b}=i^{\phi_a}$ such that $u$ is adjacent to $w$ in $\Gamma$.
Therefore, $(u,0)$ is adjacent to $(w,1)$ in $\Pi$, and so $(u,0)^\tau$ is adjacent to $(w,1)^\tau$ in $\Pi$, which implies that $(u,0)^{\tau\pi_C}\neq(w,1)^{\tau\pi_C}$. Then according to~\eqref{eq:u_a} and~\eqref{eq:w_b}, $u_a^{\phi_a}\neq w_b^{\psi_b}$. However, we deduce from $u_a=i$ and $w_b^{\psi_b}=i^{\phi_a}$ that $u_a^{\phi_a}=w_b^{\psi_b}$, a contradiction.
\end{proof}

\begin{claim}\label{claim3}
\eqref{it:itdcb} of Lemma~\ref{le:ledc} holds.
\end{claim}

\begin{proof}
Note that $\Gamma$ is bipartite with parts $U$ and $W$. Let $u\in U$, $w\in W$ and $i,j\in V(C_n)$ such that there exist edges of $\Gamma$ from $P_u$ to $Q_w$ and $L_i$ to $L_j$. We show that there exists an edge of $\Gamma$ from $P_u\cap L_i$ to $Q_w\cap L_j$. More precisely, letting
\begin{align*}
\tilde{u}&=(u_0,u_1,\ldots,u_{a-1},i,u_{a+1},\ldots,u_{c-1}),\\
\tilde{w}&=(w_0,w_1,\ldots,w_{b-1},j^{\phi_a\psi_b^{-1}},w_{b+1},\ldots,w_{d-1}),
\end{align*}
the definition of $L_i$, $P_u$ and $Q_w$ gives $\tilde{u}\in P_u\cap L_i$ and $\tilde{w}\in Q_w\cap L_j$, and we prove in the following that $\{\tilde{u},\tilde{w}\}$ is an edge of $\Gamma$.

As there is an edge, say $\{u',w'\}$, from $P_u$ to $Q_w$, where
\begin{align*}
u'&=(u_0,u_1,\ldots,u_{a-1},u'_a,u_{a+1},\ldots,u_{c-1}),\\
w'&=(w_0,w_1,\ldots,w_{b-1},w'_b,w_{b+1},\ldots,w_{d-1}),
\end{align*}
it follows that $\{(u',0)^\tau,(w',1)^\tau\}$ is an edge of $\Pi$. In particular, $\{(u',0)^{\tau\pi_\Gamma},(w',1)^{\tau\pi_\Gamma}\}$ is an edge of $\Gamma$. Since condition~\eqref{condition2} implies $(\tilde{u},0)^{\phi\pi_\Gamma}=(u',0)^{\phi\pi_\Gamma}$, we have
\[
(\tilde{u},0)^{\tau\pi_\Gamma}=(\tilde{u},0)^{\phi(\gamma,\mathrm{id})\pi_\Gamma}=(u',0)^{\phi(\gamma,\mathrm{id})\pi_\Gamma}
=(u',0)^{\tau\pi_\Gamma}.
\]
Similarly, $(\tilde{w},1)^{\tau\pi_\Gamma}=(w',1)^{\tau\pi_\Gamma}$. Therefore, $\{(\tilde{u},0)^{\tau\pi_\Gamma},(\tilde{w},1)^{\tau\pi_\Gamma}\}$ is an edge of $\Gamma$.

As there is an edge, say $\{u'',w''\}$, from $L_i$ to $L_j$, where
\begin{align*}
u''&=(u_0'',u_1'',\ldots,u_{a-1}'',i,u_{a+1}'',\ldots,u_{c-1}''),\\
w''&=(w_0'',w_1'',\ldots,w_{b-1}'',j^{\phi_a\psi_b^{-1}},w_{b+1}'',\ldots,w_{d-1}''),
\end{align*}
it follows that $\{(u'',0)^\tau,(w'',1)^\tau\}$ is an edge of $\Pi$, and so $\{(u'',0)^{\tau\pi_C},(w'',1)^{\tau\pi_C}\}$ is an edge of $C_n$. From~\eqref{eq:u_a} and~\eqref{eq:w_b} we deduce $(\tilde{u},0)^{\tau\pi_C}=(u'',0)^{\tau\pi_C}$ and $(\tilde{w},1)^{\tau\pi_C}=(w'',1)^{\tau\pi_C}$. Hence $\{(\tilde{u},0)^{\tau\pi_C},(\tilde{w},1)^{\tau\pi_C}\}$ is an edge of $C_n$. This together with the conclusion $\{(\tilde{u},0)^{\tau\pi_\Gamma},(\tilde{w},1)^{\tau\pi_\Gamma}\}\in E(\Gamma)$ yields $\{(\tilde{u},0)^\tau,(\tilde{w},1)^\tau\}\in E(\Pi)$. Consequently, $\{(\tilde{u},0),(\tilde{w},1)\}$ is an edge of $\Pi$, and so $\{\tilde{u},\tilde{w}\}$ is an edge of $\Gamma$, as required.
\end{proof}

\begin{claim}\label{claim4}
\eqref{it:itdcc} of Lemma~\ref{le:ledc} holds.
\end{claim}

\begin{proof}
Suppose that there exists an edge, say $\{u',w'\}$, from $P_u$ to $Q_w$, where
\begin{align*}
u'&=(u_0,u_1,\ldots,u_{a-1},u'_a,u_{a+1},\ldots,u_{c-1}),\\
w'&=(w_0,w_1,\ldots,w_{b-1},w'_b,w_{b+1},\ldots,w_{d-1}).
\end{align*}
We show that each $\tilde{u}\in P_u$ and $\tilde{w}\in Q_w$ satisfy $|N_\Gamma(\tilde{u})\cap Q_w|=|N_\Gamma(\tilde{w})\cap P_u|=2$.
By the second paragraph in the proof of Claim~\ref{claim3},
\[
(\tilde{u},0)^{\tau\pi_\Gamma}=(u',0)^{\tau\pi_\Gamma}\ \text{ and }\ (\tilde{w},1)^{\tau\pi_\Gamma}=(w',1)^{\tau\pi_\Gamma}
\]
with $\{(\tilde{u},0)^{\tau\pi_\Gamma},(\tilde{w},1)^{\tau\pi_\Gamma}\}\in E(\Gamma)$. According to~\eqref{eq:u_a} and~\eqref{eq:w_b}, we have $(\tilde{u},0)^{\tau\pi_C}=\tilde{u}_a^{\phi_a}$ and $(\tilde{w},1)^{\tau\pi_C}=\tilde{w}_b^{\psi_b}$. Thus, for each $\tilde{u}\in P_u$,
\[
\left|\big\{\tilde{w}\in Q_w:\{(\tilde{u},0)^\tau,(\tilde{w},1)^\tau\}\in E(\Pi)\big\}\right|
=\left|\big\{\tilde{w}_b\in V(C_n):\{\tilde{u}_a^{\phi_a},\tilde{w}_b^{\psi_b}\}\in E(C_n)\big\}\right|=2,
\]
which implies that
\[
|N_\Gamma(\tilde{u})\cap Q_w|=|N_\Pi((\tilde{u},0))\cap (Q_w\times\{1\})|=|N_\Pi((\tilde{u},0)^\tau)\cap(Q_w\times\{1\})^\tau|=2.
\]
Similarly, $|N_\Gamma(\tilde{w})\cap P_u|=2$ for each $\tilde{w}\in Q_w$. This completes the proof.
\end{proof}

\subsection{Proof of Theorem~$\ref{le:lepco}$}

Since $n$ is even, $C_n$ is bipartite with bipartition $\{U,W\}$, where
\[
U:=\{i\in V(C_n):i\text{ is even}\}\ \text{ and }\ W:=\{i\in V(C_n):i\text{ is odd}\}.
\]
If $\Gamma$ is nontrivially unstable and coprime to $C_n$, then $(\Gamma,C_n)$ is a nontrivial pair, and Theorem~\ref{thm3}\,\eqref{thm3a} implies that $(\Gamma,C_n)$ is unstable. Conversely, suppose that $(\Gamma,C_n)$ is nontrivially unstable. Then $\Gamma$ is connected, bipartite, twin-free, and coprime to $C_n$. To complete the proof, suppose for a contradiction that $\Gamma$ is unstable.

First let $n$ be a multiple of $4$. Note that $\delta\colon i\mapsto-i$ is an automorphism of $C_n$ with fixed point set $\{0,n/2\}$. In particular, $\delta$ is an automorphism of $S(C_n)$ that fixes some vertex in $U$ but not any in $W$. Moreover, note from~\eqref{eq:S(Cn)} that each connected component of $S(C_n)$ is a cycle $C_{n/2}$. Then since $C_{n/2}$ is Cartesian-prime, we conclude from Theorem~\ref{thm3}\,\eqref{thm3c} that $(\Gamma,C_n)$ is stable, a contradiction.

Now let $n=2m$ with $m$ odd. Since $\Gamma$ is connected, non-bipartite and twin-free, $\Gamma\times K_2$ is connected and twin-free.
Since $m$ is odd, we have
\[
C_n\cong K_2\times C_m,\ \ D_{2n}\cong Z_2\times D_{2m}
\]
and that $C_m$ has no decomposition $C_m\cong\Sigma\times\Delta$ with $\Delta$ of order greater than $1$. Thus, $\Gamma\times K_2$ and $C_m$ are coprime, and so Theorem~\ref{thm1} yields
\begin{align*}
\aut(\Gamma\times K_2\times C_m)&\cong\aut(\Gamma\times K_2)\times\aut(C_m)\\
&\cong\aut(\Gamma)\times\aut(K_2)\times\aut(C_m)\quad(\text{by the stability of $\Gamma$})\\
&\cong\aut(\Gamma)\times C_2\times D_{2m}\\
&\cong\aut(\Gamma)\times D_{2n}\\
&\cong\aut(\Gamma)\times\aut(C_n)\\
&\not\cong\aut(\Gamma\times C_n)\quad(\text{by the instability of $(\Gamma, C_n)$})\\
&=\aut(\Gamma\times K_2\times C_m),
\end{align*}
again a contradiction.

\subsection{Proof of Theorem~$\ref{thm2}$}

By virtue of Theorem~\ref{thm1}, we only need to prove the ``only if'' part. This can be done by constructing an unstable graph $\Gamma$ such that $(\Gamma,C_n)$ is nontrivial, where $n$ is even, and then applying Theorem~\ref{thm3}\,\eqref{thm3a}. Here, however, we present a direct proof that avoids invoking Theorem~\ref{thm3}, as it is no more complicated than the indirect approach.

Suppose that $n$ is even. Since $n\neq4$, the graph $C_n$ is twin-free. Let $\Gamma=\mathrm{Cay}(Z_8,\{\pm1,\pm2\})$. Clearly, $\Gamma$ is connected, non-bipartite and twin-free. Since the valency of $\Gamma$ is greater than that of $C_n$ and the only divisors of $|V(\Gamma)|=8$ is $1$, $2$, $4$ and $8$, it is also easy to see that $\Gamma$ is coprime to $C_n$. Consider the permutation $\tau$ on $V(\Gamma\times C_n)=Z_8\times Z_n$ defined by
\[
(u,i)^\tau=(5u+4i,i),
\]
where note that $5u+4i$ is well defined as $4n$ is divisible by $8$. Since
\begin{align*}
&\{(u,i)^\tau,(v,j)^\tau\}\in E(\Gamma\times C_n)\\
\Leftrightarrow\ &\{(5u+4i,i),(5v+4j,j)\}\in E(\Gamma\times C_n)\\
\Leftrightarrow\ &(5u+4i)-(5v+4j)\in\{\pm1,\pm2\}\text{ and }i-j\in\{\pm1\}\\
\Leftrightarrow\ &5(u-v)+4\in\{\pm1,\pm2\}\text{ and }i-j\in\{\pm1\}\\
\Leftrightarrow\ &u-v\in\{\pm1,\pm2\}\text{ and }i-j\in\{\pm1\}\\
\Leftrightarrow\ &\{(u,i),(v,j)\}\in E(\Gamma\times C_n),
\end{align*}
$\tau\in\Aut(\Gamma\times C_n)$. Moreover, since $\tau$ maps $(0,0)$ and $(0,1)$ to $(0,0)$ and $(4,1)$ respectively,   $\tau\notin\Aut(\Gamma)\times\Aut(C_n)$. Hence $(\Gamma,C_n)$ is nontrivially unstable, proving the ``only if'' part.

\section*{Acknowledgments}

Wang was supported by the National Natural Science Foundation of China (12401453), the China Postdoctoral Science Foundation (2024M751251) and the Postdoctoral Fellowship Program of CPSF (GZC20240626). Qin was supported by the National Natural Science Foundation of China (12101421).

\end{document}